\documentclass[11pt,twoside,reqno]{amsart}
\usepackage{amsmath}
\usepackage{amssymb}
\usepackage{mathrsfs}

\usepackage[mathcal]{eucal}
\usepackage{hyperref}
\usepackage[alphabetic]{amsrefs}

\newtheorem{Thm}{Theorem}
\newtheorem{Prop}[Thm]{Proposition}
\newtheorem{Lem}[Thm]{Lemma}
\newtheorem{Cor}[Thm]{Corollary}

\theoremstyle{definition}
\newtheorem{Def}[Thm]{Definition}
\newtheorem{Ex}[Thm]{Example}
\newtheorem{Rem}[Thm]{Remark}

\def\le{\leqslant}
\def\ge{\geqslant}
\def\C{{\mathbb C}}
\def\R{{\mathbb R}}
\def\Q{{\mathbb Q}}
\def\Z{{\mathbb Z}}
\def\N{{\mathbb N}}
\def\O{{\mathscr O}}
\def\D{{\mathbb D}}
\def\Om{\varOmega}
\def\Re{\operatorname{Re}}
\def\Im{\operatorname{Im}}
\let\parasymbol=\S
\def\secref#1{\parasymbol\ref{#1}}
\def\S{\varSigma}
\def\ssm{\smallsetminus}
\def\d{\mathrm d}

\def\l{\lambda}

\def\L{\varLambda}

\def\e{\varepsilon}
\def\bigland{\textstyle\bigwedge\nolimits}
\def\G{\varGamma}
\def\f{\varphi}
\def\^#1{\widehat{#1}}

\def\:{\colon}

\def\sign{\operatorname{sign}}
\def\P{\mathbb P}

\def\<{\left<}
\def\>{\right>}

\def\pd#1#2{\frac{\partial #1}{\partial #2}}
\def\Arg{\operatorname{Arg}}
\def\gap{\operatorname{gap}}
\def\mult{\operatorname{mult}}
\def\ord{\operatorname{ord}}
\def\eu{{\boldsymbol\epsilon}}
\def\~#1{\widetilde{#1}}

\begin{document}

\title{Rolle models in the real and complex world}

\author{D. Novikov, S. Yakovenko}

\begin{abstract}
Numerous problems of analysis (real and complex) and geometry (analytic, algebraic, Diophantine e.a.) can be reduced to calculation of the ``number of solutions'' of systems of equations, defined by algebraic equalities and differential equations with algebraic right hand sides (both ordinary and Pfaffian). In the purely algebraic context the paradigm is given by the B\'ezout theorem: the number of isolated solutions of a system of polynomial equations of degree $\le d$ in the $n$-dimensional space does not exceed $d^n$, the bound polynomial in $d$ and exponential in $n$. This bound is optimal: if we count solutions properly (i.e., with multiplicities, including complex solutions and solutions on the infinite hyperplane), then the equality holds.

This paradigm can be generalized for the transcendental case as described above. It turns out some counting problems admit similar bounds depending only on the degrees and dimensions, whereas other can be treated only locally, i.e., admit bounds for the number of solutions in some domains of limited size. There is only one class of counting problems, which admits global bounds, but besides the degree and dimension, the answer depends on the height of the corresponding Pfaffian system.

The unifying feature for these results is the core fact that lies at the heart of their proofs. This fact can be regarded as a variety of distant generalizations of the Rolle theorem known from the undergraduate calculus, claiming that between any two roots of a univariate differentiable function on a segment must lie a root of its derivative. We discuss generalizations of the Rolle theorem for vector-valued and complex analytic functions (none of them straightforward) and for germs of holomorphic maps.
\end{abstract}

\dedicatory{To Askold Georgievich prince Khovanskii, our Teacher and lifelong Role Model, with admiration.}

\address{Department of Mathematics, The Weizmann Institute of Science,\\ 234 Herzl Street, POB 26, Rehovot 7610001\\ Israel }
\email{\{dmitry.novikov,sergei.yakovenko\}@weizmann.ac.il}

\maketitle

\tableofcontents

\section{Rolle lemma, virgin flavor}

A real function differentiable on a compact real segment, which vanishes at both endpoints, must have a root of its derivative somewhere in the interior (even if the derivative is not continuous). This statement, known as the Rolle lemma, follows from the fact that a non-constant function must achieve both its maximum and minimum on the segment. Unless the function is constant, one of these points should be in the interior and hence the derivative must vanish there.

This absolutely elementary principle turns out to be at the origin of several very powerful techniques allowing numerous important applications\footnote{The ``baby version'' of this survey, listing the principal results available before 2015, appeared in \cite{EMS}: the latter can be considered as a gentle introduction into the area.}.

 \subsection{First year calculus revisited}\label{sec:rolle-1}
We start with the simplest modifications. For a function $f\:[0,1]\to\R$ denote by $Z(f)\le\infty$ the number of its geometrically distinct zeros (roots) on $[0,1]$.
\begin{Prop}[Rolle inequality]
 If $f$ is differentiable, then $Z(f')\ge Z(f)-1$, hence $Z(f)\le Z(f')+1$. In particular, if $Z(f')<+\infty$, then $Z(f)<+\infty$ as well.\qed
\end{Prop}
\begin{Prop}[Rolle inequality, periodic case]
If $f\:\R\to\R$ is a one-periodic differentiable function, then $Z(f)\le Z(f')$.\qed
\end{Prop}
One can combine these two different cases into one inequality.
\begin{Prop}\label{prop:rolle-refined}
Assume that $f$ is differentiable on $[0,1]$ and neither $f$ nor $f'$ vanish at the endpoints $0,1$. Then $Z(f)\le Z(f')+\f(t)\big|_{t=0}^{t=1}$, where $\f(t)=\tfrac12\bigl|\sign f(t)-\sign f'(t)\bigr|\in\{0,1\}$.
\end{Prop}
\begin{proof}
If at $t=0$ $f$ and $f'$ have the same sign, then between $0$ and the smallest root of $f$ must be a root of $f'$. A similar argument (mutatis mutandis) applies to the interval between the largest root of $f$ and $t=1$.
\end{proof}

Consider now the case where the function $f$ is real analytic on $[0,1]$, including the endpoints. This allows to introduce the \emph{multiplicity} of roots of $f$ and all its derivatives (which may well be infinite). Denote by $N(f)<+ \infty$ the number of roots of $f$ \emph{counted with their multiplicities}.
\begin{Prop}\label{prop:real-rolle}
$N(f)\le N(f')+1$. \qed
\end{Prop}
The proof follows from the fact that if $f$ has a root of multiplicity $\mu\ge 2$ at some point, then $f'$ has a root of multiplicity $\mu-1$. Periodic and synthetic versions of this inequality are also true.

The numbers $N(\cdot),Z(\cdot)$ considered as functionals, satisfy a \emph{multiplicative} triangle inequality.

\begin{Prop}
$$
|Z(f)-Z(g)|\le Z(fg)\le Z(f)+Z(g)\qquad N(fg)=N(f)+N(g).\qed
$$
\end{Prop}

\begin{Small}  \subsection{Rolle inequality and Descartes law}
An immediate application of the above inequalities is an upper bound for the number of isolated roots of polynomials which gives the answer not in terms of their degree, but rather in terms of the number of their nonzero coefficients. The difference becomes critical for sparce polynomials of high degree having only a few nonzero coefficients. The term \emph{fewnomials} was suggested by A. Khovanskii and it is now firmly rooted in the tradition, though the word \emph{olygonomials} probably would be better off stylistically.
\begin{Prop}
 The number of positive roots of a Laurent polynomial $p(t)=\sum_{\alpha\in A} c_\alpha t^\alpha$, where $A\subset\Z$ a finite set with $|A|=n$ elements, does not exceed $n-1$. The bound is sharp.
\end{Prop}
\begin{proof}
Without loss of generality we may assume that $A\subseteq\Z_+$ and $0\in A$, multiplying $p$ by a suitable power $t^\mu$ if necessary. This multiplication does not change the number of \emph{positive} roots of $p$. The derivative is again a fewnomial with the new index set $A'=A\ssm\{0\}$ which has at most $n-1$ distinct indices. This  allows to use induction in the number of terms. After $n-1$ derivations we get a nontrivial fewnomial with $A=\{0\}$ which has no positive roots. Inductive application of the Rolle inequality proves the claim.
\end{proof}
In fact, using more refined Proposition~\ref{prop:rolle-refined}, one can prove that the number of positive roots is bounded by the number of sign changes in the sequence of nonzero coefficients.


\subsection{Main building block of elementary Fewnomial theory}
The most direct multidimensional analogue of the Rolle inequality deals with smooth curves and their intersection with \emph{Pfaffian hypersurfaces} meeting certain topological conditions \cite{fewnomials}.

Let $\gamma\:(a,b)\to\R^n$ be a smooth (parametrized) curve and $\G$ a real hypersurface, \emph{not necessarily connected}. We say that $\G$ is Pfaffian, if there exists a Pfaffian 1-form $\omega$ which vanishes on $T\G=\bigcup_{a\in\G}T_a\G$ (i.e., $\G$ is an integral surface of $\omega$, although in general $\omega$ might violate integrability conditions and have no other integral hypersurfaces). We say that $\G$ is a \emph{separating hypersurface}, if its a topological boundary of a domain $D$ in $\R^n$ and $\omega$ takes positive values on all outbound vectors transversal to $\G$.

A point $a=\gamma(t_*)$ is called the point of tangency between $\gamma$ and $\omega$, if $\omega$ vanishes on the velocity vector $\dot\gamma=\frac{\d }{\d t}\gamma(t_*)$. The set of tangency points will be denoted $\{\gamma\parallel\omega\}$.

\begin{Thm}
If $\G$ is a separating solution of the Pfaffian equation $\omega=0$, then the number of points of intersection $\#\{\gamma\cap\G\}\le\#\{\gamma\parallel\omega\}+1$.
\end{Thm}
\begin{proof}
Without loss of generality assume that $\gamma$ crosses $\G$ transversally so that $\{\gamma\cap\G\}$ consists of isolated points. Then the values $\omega(\dot\gamma)$ at these points have alternating signs: indeed, a trajectory that entered $D$ can cross its boundary $\G=\partial D$ only when leaving it, and vice versa. Thus between any two consecutive intersections between $\gamma$ and $\G$ there must be at leat one tangency point.
\end{proof}

Note that, say, if $\gamma$ is an algebraic curve and the form $\omega$ is polynomial, then the tangency set $\{\gamma\parallel\omega\}$ is algebraic; if it consists of isolated points, then their number can be bounded from above by the B\'ezout theorem. This allows to ``extend'' the B\'ezout theorem to some transcendental cases.

\begin{Ex}
Consider a polynomial 2-form on the plane $\omega=P(x,y)\,\d x+Q(x,y)\,\d y$ of degree $n$. Isolated compact integral manifolds (curves) for $\{\omega=0\}$ are called \emph{limit cycles} of the corresponding differential equation. Each such cycle is a separating hypersurface (actually, a planar curve). Moreover, some of them can be combined into non-connected separating  hypersurface (it depends on the orientation of limit cycles.

One can easily prove that any separating hypersurface  $\G$ can transversally intersect an algebraic curve of degree $m$ by no more than $m(n+m)$ points. This implies that any separating solution has no more than $2n^2$ connected components. Unfortunately, this does not imply any bound on the number of limit cycles of $\omega$, since the topological constraints exclude ``most'' of limit cycles from simultaneous inclusion in the separating hypersurface $\G$.
\end{Ex}

This construction can be iterated after appropriate precautions, developing into the Fewnomial theory. Speaking very loosely, there is a way to reduce a ``system of Pfaffian equations'' $\omega_1=0, \dots, \omega_k=0$, complemented by algebraic equations $P_{k+1}=0, \dots, P_n=0$ in $\R^n$ to another system (actually, several systems) involving $k-1$ Pfaffian ``equations'' and $n-k+1$ algebraic ones. For this, the Pfaffian equations (actually, the increasing chain of integral manifolds $\G_1,\dots,\G_k$) must satisfy certain inductive topological condition to form a \emph{Pfaffian chain}. This direction constitutes the main body of the book \cite{fewnomials} and we don't explore it here. \par
\end{Small}

\section{Rolle theorem and real ODE's}

 \subsection{De la Vall\'ee Poussin theorem and higher order equations}
The Rolle lemma(s) from \secref{sec:rolle-1} can be iterated so that it applies to linear ordinary differential operators of higher order, not just the single derivation.

Denote by $\|\,\cdot\,\|$ the sup-norm on the segment $[0,\ell]\subseteq\R$ (we will consider only smooth functions). It is well-known that the derivation is an unbounded operator in all practical senses, yet no opposite inequality of the sort $\|f\|\le C\|f'\|$ can exist simply because it will be eventually violated if $f$ is replaced by $f+c$ with  $c\in\R$ sufficiently large. Things change if we insist that $f$ is vanishing somewhere on $[0,\ell]$.

\begin{Lem}\label{lem:symplex}
If $f$ has a root on $[0,\ell]$, then $\|f\|\le \ell\|f'\|$.

More generally, if $f$ has $n+1\ge 1$ roots on $[0,\ell]$, counted with multiplicities, then $\|f\|\le \frac{\ell^n}{n!}\|f^{(n)}\|$.
\end{Lem}
\begin{proof}
The first claim follows immediately from the Newton--Leibniz formula
$$
\forall t\in [0,\ell]\qquad f(t)=f(t_*)+\int_{t_*}^t f'(s)\,\d s=\int_{t_*}^t f'(s)\,\d s,
$$
if we choose $t_*\in[0,\ell]$ at the root of $f$ and majorize the integral by $\ell\, \|f'\|$.

The second statement is obtained by iteration of the first claim. By the Rolle theorem, subsequent derivatives $f',f'',\dots, f^{(n)}$ must have at least $n, n-1,\dots, 1$ root on $[0,\ell]$ respectively. The iterated Newton--Leibniz formula yields for $f$ an integral of $f^{(n)}$ over a domain in $\R^n$ which lies inside the standard symplex $\{0\le t_0\le t_1\le\cdots\le t_n\le 1\}\subseteq\R^{n+1}$. Thus $\|f\|$ does not exceed  the volume $\ell^n/n!$ of the symplex  multiplied by the norm $\|f^{(n)}\|$.
\end{proof}

This Lemma implies an immediate corollary.

\begin{Prop}[de la Vall\'ee Poussin theorem, 1929]\label{prop:dlVP}
Consider a homogeneous linear ordinary differential equation of the form
\begin{equation}\label{lode}
  y^{(n)}+a_1(t)\,y^{(n-1)}+\cdots+a_{n-1}(t)\,y'+a_n(t)\,y=0, \qquad t\in[0,\ell]\Subset\R
\end{equation}
with continuous coefficients $a_1,\dots,a_n$ which are bounded, $\|a_k\|\le A_k<+\infty$.

Assume that the length $\ell$ is small enough (compared to the magnitude of the coefficients) to satisfy the inequality
\begin{equation}\label{lode-bounds}
  \sum_{i=1}^n A_k\frac{\ell^k}{k!}<1.
\end{equation}
Then any solution of the equation \eqref{lode} has no more than $n-1$ isolated root on $[0,\ell]$.
\end{Prop}

\begin{proof}
Let $f$ be an arbitrary solution. If $\|f^{(n)}\|=0$, then $f$ is a polynomial of degree $\le n-1$ and the claim is trivial. Otherwise without loss of generality one may assume that $\|f^{(n)}\|=1$ (since the equation is homogeneous) and then the leading term of the identity \eqref{lode} after substitution $y=f(t)$ overtakes the sum of all non-principal terms at the point where the maximum $1=\max_t |f^{(n)}(t)|=\|f^{(n)}\|$ is achieved by Lemma~\ref{lem:symplex}.
\end{proof}
The bound for the number of roots is optimal. Indeed, in any $n$-dimen\-sional space of real functions on any $[0,\ell]$ any $n-1$ points can be assigned for roots of a nontrivial function of this subspace.

\subsection{Real meandering theorem}
Proposition~\ref{prop:dlVP} allows to place effective upper bounds on the number of isolated zeros of \emph{linear} differential equations with explicitly bounded coefficients on any bounded interval. To do this, the interval should be subdivided into sufficiently small parts satisfying \eqref{lode-bounds}. Somewhat unexpectedly the linearity assumption can be replaced by mere polynomiality assumption.

Consider a polynomial vector field in $\R^n$, defined by a system of polynomial ordinary differential equations
\begin{equation}\label{pode}
  \dot x_i=v_i(x)=\sum_{i,\alpha}c_{i\alpha}x^\alpha,\quad i=1,\dots,n,\quad \alpha=(\alpha_1,\dots,\alpha_n)\in\Delta\Subset\Z_+^n.
\end{equation}
Besides the dimension $n$, this field is also characterized by its degree $d=\max_\Delta|\alpha|$, where $|\alpha|=|\alpha_1+\cdots+\alpha_n|$. Outside of the singular locus $\S=\{x\in\R^n:v(x)=0\}$ trajectories of this vector field are smooth real analytic curves. We are interested in an upper bound for the number of isolated intersections between ``bounded'' pieces of these curves and affine hyperplanes.

Already simplest (linear) examples show that this bound, besides the ``size'' of a piece, must depend also on its proximity to infinity and the ``height'' of the equation (magnitude of the coefficients $|c_{i\alpha}|$. To minimize the number of independent parameters, we assume that all coefficients are bounded in the absolute value by the same number $R>0$:
\begin{equation}\label{magn}
  |q_i|\le R,\quad |c_{i\alpha}|\le R\qquad \forall i=1,\dots,n,\quad |\alpha|\le d,
\end{equation}
where $q=(q_1,\dots,q_n)\in\R^n$ is the initial value $q=\gamma(0)$ for \eqref{pode}.

%
%
%

Then for any sufficiently small $\delta>0$  one can consider the analytic integral curve $\gamma_{q,\delta}\:(-\delta,\delta)\to\R^n$ of this equation with the initial condition $\gamma(0)=q$ and study the number of its isolated intersections $\#\{\gamma_{q,\delta}\cap \varPi\}$ with an arbitrary affine hyperplane $\varPi\subseteq\R^n$.

It turns out, that this question can be reduced to Proposition~\ref{prop:dlVP}.

\begin{Thm}[Novikov and Yakovenko \cite{meandering}]\label{thm:meandering}
	There exists a natural number $\nu$ depending only on $n,d$ and there exists a small $\delta>0$ which depends on $n,d,R$, such that $\#\{\gamma_{q,\delta}\cap\varPi\}\le \nu$ for any affine hyperplane $\varPi$.
	
	The values of $\nu,\delta$ are explicitly bounded from above and from below respectively:
	$$
	\nu\le d^{2^{O(n^2\ln n)}}, \qquad \delta>R^{-2^{\operatorname{Poly}(n,d)}},
	$$
	where $\operatorname{Poly}(n,d)$ stands for an explicit polynomial in $n,d$.
\end{Thm}

\begin{Small}
\begin{proof}[Idea of the proof]
Consider an arbitrary affine hyperplane defined by an affine equation $u_0=0$ in $\R^n$, $u_0\in\R[x]=\R[x_1,\dots,x_n],\ \deg u_0=1$. The vector field \eqref{pode} defines the derivation $D=D_v$ of the algebra $\R[x]$, $Du=\sum_{i=1}^n \pd u{x_i}\, v_i$.
Consider the ascending chain of polynomial ideals generated by consecutive derivations,
\begin{equation}\label{chain}
  \<u_0\>\subseteq\<u_0, u_1\>\subseteq\<u_0, u_1, u_2\>\subseteq \cdots,\quad u_{i+1}=Du_i,\quad i=0,1,2,\dots
\end{equation}
This chain must stabilize at a certain step $\nu$, which means that there exists an identity\footnote{Since the chain is obtained by adjoining the consecutive derivatives, from the moment of the first stabilization the chain stabilizes forever.}
\begin{equation}\label{plode}
  u_\nu=\sum_{i=1}^\nu h_iu_{\nu-i},\qquad h_i\in\R[x].
\end{equation}
Restricting this identity on any curve $\gamma_{p,\delta}$ parameterized by $t$, we obtain a linear identity between the derivatives $u_0^{(i)}$ with the coefficients $a_i(t)=h_i\bigl|_{\gamma_{p,\delta}}$, that is, a linear ``equation'' \eqref{lode} of order $\nu$.

\begin{Rem}
This is the key idea when working with what later will be called the Noetherian rings, see~\secref{sec:GK-mult} below: all transcendental objects are obtained by restriction of appropriate polynomial or algebraic objects in some ambient affine space on (transcendental) integral manifolds (i.e., integral curves) defined by algebraic differential equations, and analytic operations like, e.g., differentiation amount to algebraic operations with the algebraic data only. The ambient algebraicity allows to use the powerful toolbox of commutative algebra and algebraic geometry to estimate the complexity of the algebraic part of the description, while tools from the classical analysis are used to transform this information into answers for different counting problems.
\end{Rem}

To complete the proof, it remains to find explicit bounds for $\nu$ and $\|a_i\|$ to apply Proposition~\ref{prop:dlVP}. As it turns out, the most difficult part is the effective Noetherianity, that is, finding the bound of the length $\nu$. Once it is known, one can use the polynomial nature of the input data. The degrees of the polynomials $u_i$ grow linearly with the growth of $i$ and hence can be explicitly bounded if the length $\nu$ is known. The growth of the norms $\|u_i\|$ can also be explicitly controlled. Then the polynomials $h_i$, $i=1,\dots,\nu$ can be found by solving a system of linear algebraic equations with known (and explicitly bounded) matrix of coefficients. Restricting polynomials $h_i$ of bounded norms on the piece of integral curve of known bounded size will yield the bounds for the scalar coefficients $\|a_i\|$ in  Proposition~\ref{prop:dlVP}.

These computations are straightforward except for one step: knowing an \emph{upper bound} for the coefficients of a system of linear algebraic equations with some matrix $M$ (in general, non-square)  does not imply any bound on the vector of its solutions, even if this system is known to be compatible (solvable), that is, of the appropriate rank. Indeed, the corresponding (nonzero) minors $\det M_\alpha$ may be arbitrarily close to zero, and division by the respective ``small denominators'' may result in unpredictably large numbers.

The situation changes completely if the coefficients of the matrix $M$ are \emph{integer} numbers. Then each minor $\det M_\alpha$ can be either zero or an nonzero integer, that is, at least one in the absolute value. In such case no ``small denominators'' can appear and the required upper bound for solutions of the linear system becomes easily computable.

How to achieve this integrality? Declare all parameters (the coefficients $c_{i\alpha}$ of the vector field, the coefficients of the affine function $u_0$, the initial point $p$) as \emph{new dependent variables}, governed by the ``differential equations'' of the form ``derivative of the variable is identically zero''. This increases dramatically the dimension of the problem: one has to add extra variables for coefficients before \emph{all monomials} of degree $\le d$ in $n$ original variables. Yet in the result we obtain a polynomial vector field of known degree with coefficients being only zeros and ones, and the linear form $u_0$ will also become a quadratic form in the new variables, also with $\{0,1\}$-coefficients. This means that the initial data are all defined over $\{0,1\}$, and hence the chain of ideals \eqref{chain} is spanned by polynomials with integer coefficients.

The norms $\|h_i\|$ will become algebraic functions in these new variables depending only on the data $n,d$ and $\nu$ which is already known. Their restrictions on the ball/box of radius $R$  will be explicitly bounded polynomials of $R$.
\end{proof}

\begin{Rem}
The effective Noetherianity is by no means an easy thing, although there is an explicit algorithm for computing the maximal length of the chain. In \cite{montreal} it is explained why for a general chain of ideal generated by polynomials of growing degrees the the length of the chain could grow as the so called \emph{Ackermann generalized exponential} of the dimension $n$, that is, the function which grows asymptotically faster than any elementary or even primitive recursive function of $n$. This is in contrast with the ``only'' double exponential bound for $\nu$ in Theorem~\ref{thm:meandering}. It is the fact that the chain is obtained by adjoining consecutive derivations which forces the chain to stabilize abnormally fast.
\end{Rem}

\par
\end{Small}

Later, when explicitly bounding the number of intersections, we will explicitly assume that all objects (vector fields, differential equations e.a.) involved in the construction, are defined over $\Q$ and have explicitly bounded height.

\begin{Def}\label{def:defQ}
We say that a rational function is \emph{defined over $\Q$} on $\C^n$, if it belongs to the field $\Q(x_1,\dots,x_n)$. Its \emph{height} is the maximal natural number required to write down this function explicitly, using only irreducible rational fractions.

An object (vector field, scalar or matrix Pfaffian form e.a.) is defined over $\Q$, if all its coordinates are rational functions defined over $\Q$. Then the \emph{height} of is defined as the maximum of heights of individual entries.
\end{Def}

If the point $q$ is \emph{nonsingular}, Theorem~\ref{thm:meandering} gives a lower bound for the size of segments of phase trajectories of a polynomial vector field which, while being transcendental, from the point of view of intersections behave as an algebraic curve of degree $\le\nu$.

\begin{Small}
\begin{Rem}\label{rem:1-overQ}
If a polynomial vector field of known degree $d$ is defined over $\Q$ and has height $s$, then the magnitude of the coefficients $|c_i\alpha|$ of this equation is bounded by $s$ and hence the height of the polynomials $u_i$ spanning the chain of ideals grows exponentially in $i$ and polynomially in $s,d$. Thus the resulting system of linear algebraic equations will have the matrix $M$ and the right hand side also defined over $\Q$ and of known bounded height, which implies an upper bound on the magnitude of its solutions (guaranteed to exist). This obliterates the need to introduce the artificial variables for $c_{i\alpha}$ and hence the dimension $n$ and the degree $d$ do not mix with each other. The overall answer for $\delta$ will be single exponential in $\log R, \log s$ and $d$, and double exponential in $n$. In other words, one can improve the dependence on $d$ from double exponential to single exponential. This can be compared with the bounds obtained in Theorem~\ref{thm:BiPi}.
\end{Rem}
\par\end{Small}

\subsection{Maximal tangency order and Gabrielov--Khovanskii theorem}
Theorem~\ref{thm:meandering}, among other things, implies that the \emph{maximal order of tangency} between an integral curve of a polynomial vector field and an algebraic hypersurface  is bounded by an expression that is polynomial in the degree $d$ of the field/hypersurface but doubly exponential in the dimension $n$ of the ambient space. This bound for tangency can be considerably improved, as the following theorem by A.~Gabri\`elov and A.~Khovanskii \cite{GK} shows.

\begin{Thm}\label{thm:GJ}
Let $\gamma$ be an integral trajectory of a polynomial vector field of degree $d$ in $\R^n$ nonvanishing at the origin, and $\varPi=\{P(y)=0\}\subset\R^n$ be an algebraic hypersurface passing through the origin and defined by a reduced polynomial $P\in\R[y_1,\dots, y_n]$, $\deg P\le d$. If the intersection $\gamma\cap\varPi$ at the origin is isolated, then the order of tangency $\mu=\ord_{0}P\big|_\gamma$ is a finite number, bounded by an expression polynomial in $d$ and simple exponential in $n$.
\end{Thm}

\begin{Small}
The proof of this result is achieved by a very elegant construction typical for the Singularity theory. In what follows we explain the main ideas of the proof in \cite{GK}.

First, without loss of generality everything can be complexified: we have a neighborhood of the origin $(\C^n,0)$, a complex analytic vector field nonsingular at $0$, and a polynomial hypersurface $\Pi$.

Since the origin is nonsingular, the vector field can be \emph{rectified} in this neighborhood, that is, without loss of generality we can assume that a local coordinate system $(x,t)\in(\C^{n-1},0)\times(\C^1,0)$ can be chosen instead of $y\in (\C^n,0)$ in such a manner that the vector field becomes $\partial/\partial t$, its trajectories we call vertical lines, and a well-defined projection $\pi\:(x,t)\mapsto x$ is well-defined. The hypersurface $\varPi\in(\C^{n},0)$ projects down onto the base $B=(\C^{n-1},0)$ and $\pi^{-1}(0)$ is the vertical line wich is tangent to $\varPi$ with finite order $\mu<+\infty$. From the Weierstrass preparation theorem, for any point $x\in B$ sufficiently close to the origin the intersection $\pi^{-1}(x)$ with $\varPi$ consists of exactly $\mu$ (complex) points, when \emph{counted with their multiplicities}. Yet the number of geometrically distinct points may vary from 1 to $\mu$.

Denote by $Z_i$ the set of points $(x,t)\in\varPi\subseteq(\C^n,0)$, at which the vertical line is tangent to $\varPi$ of order $\ge i$ for $i=1,\dots,\mu$ (the transversal intersection corresponds to $i=1$. If $\varPi$ is locally defined by an analytic equation $F(x,t)=0$, then
\begin{equation}\label{milnor}
\begin{gathered}
 Z_i=\left\{F(x,t)=\frac{\partial F}{\partial t}(x,t)=\frac{\partial^2 F}{\partial t^2}(x,t)=\cdots=\frac{\partial^{i-1} F}{\partial t^{i-1}}(x,t)=0\right\}
 \\
 \varPi= Z_1\supseteq Z_2\supseteq\cdots\supseteq Z_{\mu}\owns \{0\}.
\end{gathered}
\end{equation}
Denote by $\nu_i=\nu_i(x)$ the number of geometrically distinct points of the intersection $\pi^{-1}(x)\cap Z_i$, so that $\nu_\mu(0)=1$ and for a generic point $x\in B$ we have $\nu_2(x)=\cdots=\nu_\mu(x)=0$ by the Sard theorem (critical values of the restriction $\pi\bigl|_\varPi$ are of measure zero). Then we have
\begin{equation}\label{count}
 \sum_{i=1}^\mu \nu_i(x)= \operatorname{const}=\mu.
\end{equation}
Indeed, the left hand side is the number of preimages in $\pi^{-1}(x)\cap\varPi$, counted with their multiplicities, expressed as the sum of multiplicities of geometrically distinct points.

The identity \eqref{count} can be ``integrated\footnote{The Euler characteristic, defined for \emph{closed} tame topological spaces as the alternating sum of the number of simplices of different dimensions, is additive: $\chi(M\cup N)=\chi(M)+\chi(N)-\chi(M\cap N)$, which allows to develop a (mostly symbolic) ``integration theory'' for $\chi$ used as a finitely additive measure. In particular, an analog of the Fubini theorem can be formulated and holds for functions tame enough, see \cite{viro} for details. } over the Euler characteristics'' $\chi$, applying a formal construction using the additivity of the Euler characteristic, to produce the equality
\begin{equation}\label{euler-int}
 \sum _{i=1}^\mu \chi(Z_i)=\mu\cdot \chi(B)=\mu.
\end{equation}

The formula \eqref{count} by itself does not allow yet to conclude anything about $\mu$ (it enters in both sides of the equality). But one can perturb the hypersurface  $\varPi$ so that dimensions of the loci $Z_1,\dots,Z_\mu$ will take their \emph{generic} values. If $\varPi$ is a \emph{generic} hypersurface in $\C^n$, then the codimension of $Z_i$ in $\varPi$ is given by the \emph{number $i$ of equations} in \eqref{milnor}: a generic point on $\varPi$ is in $Z_1$, the tangency occurs on $Z_1$ which has codimension 1 in $\Pi$, double tangency on $Z_2$ of codimension 3 \emph{etc}. At the end of the sequence we see that $\operatorname{codim}_\varPi Z_{n}=n-1=\dim\varPi$, and for $i> n$ we would have $\operatorname{codim}Z_{i}>\dim\varPi$. This means that the corresponding loci must be empty, and the corresponding left hand side in \eqref{euler-int} would take the sum in which the upper limit is $n-1$ and not $\mu$.

Of course, the original hypersurface $\varPi$ (polynomial in the initial coordinates defined by the equation $P(y)=0$, $y\in\C^n$) \emph{may be non-generic}. Yet by the Thom's transversality theorem, one can find a small perturbation of the form $P_\e(y)=P(y)+\e Q(y)$ with $\deg Q\le \deg P$ and $\e\in (\R,0)$ sufficiently small, in such a way that $\varPi_\e$ would be generic in the above sense for all $\e\ne 0$. (The perturbation parameter $\e$ must be chosen \emph{very} small, depending on the size of the neighborhoods in which the formula \eqref{euler-int} is valid).

But then for any such small $\e$ we will have the equality
\begin{equation}\label{euler-poly}
 \sum_{i=1}^n \chi (Z_{i,\e})=\mu,
\end{equation}
where $Z_{i,\e}$ are the loci constructed starting from $P_\e\in\R[y]$ rather than for $P$. Their definition in the invariant terms requires the Lie derivation $L_v$ along the polynomial vector field $v$:
\begin{equation}\label{tangent-loci}
Z_{i,\e}=\left\{P_\e=0,\ L_v P_\e=0,\ L_v^2 P_\e=0,\ \dots\ ,L_v^{i-1}P_\e=0\right\}.
\end{equation}
Note that the sets $Z_{i,\e}$ are all \emph{real algebraic}, hence tame: their topological characteristics can be explicitly bounded in terms of the dimension, the degrees of the polynomial equations and their number, see \cite{milnor}.  Knowing the degree $d=\deg P=\deg P_\e$ and $\deg v$, the dimension $n$, and noting that the number of non-void $Z_{i,\e}$ for $\e\ne 0$ is at most $n-1$,  one can:
\begin{enumerate}
 \item explicitly bound the degrees of all the Lie derivatives above,
 \item estimate the Euler characterstics $\chi(Z_{i,\e})$ for all $\e\ne0$.
\end{enumerate}
These estimates are fairly accurate as functions of $d$ and $n$: they are polynomial in $d$ and simply exponential in $n$. This gives a similar upper bound for $\mu$, and the corresponding bounds can be explicitly written down, see \cite{GK}.

\begin{Rem}
The upper bound of \cite{GK} is $O(d^{2n})$, $d=\deg P$ (recall that we assume $P$ to be reduced, square-free), which falls short of $O(d^n)$ necessary for transcendental number theory applications. On the other hand, the upper bounds appearing in the latter context, like in \cite{nesterenko}, are not explicit in the degree of the vector field $v$ and in dimension $n$. In \cite{Binya-Morse}  the upper bound for the Euler characteristic of $Z_{i,\epsilon}$ was obtained using the complex Morse theory instead of the real one. This led to explicit upper bounds with asymptotics as required,  improving the  main result of \cite{GK} as well as all its predecessors.
\end{Rem}
\par
\end{Small}

\subsection{Oscillation of curves in the Euclidean space}
The previous results still do not answer the following apparently simple question. Consider a $C^\infty$-smooth spatial curve $\gamma\:[0,\ell]\to\R^n$, $t\mapsto x(t)$ in the \emph{Euclidean} space $\R^n$, and its ``derivative'', the velocity hodograph $\gamma'\:t\mapsto \dot x(t)=\frac{\d }{\d t}x(t)$. Is there a numeric measure of the ``oscillatory behavior'' (whatever this may mean) for spatial curves, which increases in a controllable way when passing from the velocity hodograph to the curve itself? Clearly, the answer would depend on how we define the oscillatory behavior in a quantitative way.

\subsubsection{Rolle theorem in $\R^n$}
One of the most natural ways is to measure the rotation of a curve around a point (or, perhaps, more generally, around an affine subspace in $\R^n$ disjoint from this curve). Assume that a smooth curve $\gamma$ avoids the origin $0\in\R^n$, that is, $\|x(t)\|\ne0$ for $t\in[0,\ell]$. Then one can define its projection on the unit sphere $\mathbb S^{n-1}_1=\{\|x\|=1\}$ as the spherical curve $S\gamma(t)=\frac{x(t)}{\|x(t)\|}$ which also will be smooth and hence has a finite length $\|S\gamma\|$. This length can be interpreted as the numerical measure of rotation of $\gamma$ around the origin. We assume that the velocity $\dot x(t)$ is nonvanishing (this guarantees that $\gamma$ is smooth), that is, the corresponding hodograph $\gamma'$ also avoids the origin and its spherical projection $S\gamma'$ has finite length. How $\|S\gamma\|$ compares with $\|S\gamma'\|$?

\begin{Thm}[Khovanskii and Yakovenko, \cite{ky-95}]\label{thm:rolle-n}
\begin{equation}\label{ky-95}
  \|S\gamma\|\le \|S\gamma'\|-\operatorname{dist}(S\gamma(t),S\gamma'(t))\big|_{t=0}^{t=\ell},
\end{equation}
where $\operatorname{dist}(\cdot,\cdot)$ is the spherical distance on $\mathbb S_1^{n-1}$.

In particular,
\begin{enumerate}
  \item if $\gamma$ is closed, $\gamma(0)=\gamma(\ell)$, $\gamma'(0)=\gamma'(\ell)$, then $\|S\gamma\|\le \|S\gamma'\|$,
  \item in any case $$\|S\gamma\|\le\|S\gamma'\|+2\pi.$$
\end{enumerate}
\end{Thm}
The second statement follows from \eqref{ky-95} since the spherical distance between any two points less or equal to $\pi$.

\begin{Small}
\begin{proof}[The idea of two different proofs]
The first proof is based on the following observation. If we consider two spherical curves $S\gamma\:t\mapsto s(t)$ and $S\gamma'\:t\mapsto s'(t)$, then the (spherical) velocity vector of $s$ is always tangent to the geodesic (large circle arc) connecting $s(t)$ with $s'(t)$. To see this, it is enough to consider the 2-dimensional section of $\mathbb S_1^{n-1}$ containing vectors $s(t)$ and $s'(t)$. In other words, the point $s(t)$ (spider) pursues optimally the point $s'(t)$ (fly), and the distance between them is decreasing no faster than $\|\dot s(t)\|-\|\dot s'(t)\|$ (the difference may well be negative). Integrating this inequality, we arrive at \eqref{ky-95}.

Another way to prove \eqref{ky-95} is the Buffon needle principle. According to one of the versions of this principle, the length of the spherical curve is equal to the average number of its intersections with a random large circle (the equator). More precisely, let $\xi\in\mathbb A_1^{n-1}$ a random vector uniformly distributed over the unit sphere and $\varPi_\xi=\{x\in\R^n\:\<\xi,x\>=0\}$ the linear hyperplane which cuts $\mathbb S_1^{n-1}$ by the random large circle. The number of intersections of any curve $\gamma$ with $\varPi$ is tautologically the same as the number of intersections of its spherical projection $S\gamma$ with the large circle. The Buffon needle principle says that
$$
\|S\gamma\|=\frac\pi{|\mathbb S_1^{n-1}|}\int_{\mathbb S_1^{n-1}}\#\{\varPi_\xi,\gamma\}\,\d \sigma(\xi),
$$
where $\d \sigma$ is the Lebesgue $(n-1)$-measure on the sphere and $|\mathbb S_1^{n-1}|$ the total volume.

For the same reasons the spherical distance between any two points, equal to the length of the ``straight'' arc connecting these points, is proportional to the probability of a random hyperplane to separate these two points. Now for each $\xi$ we may consider the scalar function $f_\xi(t)=\<\xi,x(t)\>$. Application of the Proposition~\ref{prop:rolle-refined} completes the second proof of the Theorem.
\end{proof}
\par\end{Small}

 \subsection{Voorhoeve index}
For $n=2$ the above result (for closed curves) can be reformulated in terms of a complex variable in such a way that the connection with the aboriginal Rolle inequality becomes fully transparent, cf.~with \cite{voor}.

Let $U\subseteq\C$ be a bounded domain with a smooth (or piecewise smooth) boundary $S=\partial U$, with the natural parametrization $[0,\ell]\owns t\mapsto z(t)\in S$, $|\dot z(t)|\equiv1$. Let $f$ be a holomorphic function defined in some neighborhood of $\partial U$ and nonvanishing there. Then we can compare the rotation of the curve $f(z(t))$, the image $f(\partial U)$, with that of its velocity $t\mapsto f'(z(t))\cdot \dot z(t)$, where $f'$ is the complex derivative of the function $z\mapsto f(z)$.

Rotation of a complex-valued function $g(t)$ of a real argument around the origin is equal to the absolute variation of its argument
$$
  V(g)\bigl|_0^\ell=\displaystyle\int_0^\ell\left|\frac{\d  \Arg  g(t)}{\d t}\right|\,\d t.
$$
The argument of the product $\Arg \bigl( f'(z(t))\cdot \dot z(t)\bigr)$ is the sum of arguments. Thus, applying Theorem~\ref{thm:rolle-n}, we conclude that
$$
 V(f)\big|_0^\ell \le V(f')\big|_0^\ell f+V(\dot z)\big|_0^\ell.
$$
The last term by definition is the absolute integral curvature of the boundary $\partial U$. If $U$ is convex, then it is equal to $2\pi$.

\begin{Def}
The \emph{Voorhoeve index} $V_S(g)$ of a complex function $g$ holomorphic in a neighborhood of a close curve $S\subseteq \C$ is the absolute variation of argument of $f(z)$ along this curve.
\end{Def}
By construction, the Voorhoeve index is always greater or equal to the topological index of $f(S)$ around the origin.

\begin{Prop}
If $f$ extends holomorphically inside $U$, then
$$
V_{\partial U}(f)\ge N_U(f)=\#\{z\in U\:f(z)=0\},
$$ thus the Voorhoeve index majorizes the number $2\pi N(f)$ of isolated zeros counted with multiplicities.
\end{Prop}

The immediate analog of Proposition~\ref{prop:real-rolle} now takes almost literally the same form.

\begin{Thm}\label{thm:rolle-c}
If $f$ is holomorphic in a bounded convex domain $U\subseteq\C$, then
$$
 V(f)\le V(f')+1.
$$
For non-convex domain 1 should be replaced by the absolute integral curvature of the boundary $\partial U$, divided by $2\pi$.
\end{Thm}

Because the identity $\Arg(uv)=\Arg u+\Arg v$, the Voorhoeve index satisfies the triangle inequality.

\begin{Prop}
For any two functions $f,g$ holomorphic on the same closed curve,
\begin{equation}\label{voor-triang}
|V(f)-V(g)|\le V(fg)\le V(f)+f(g).\qed
\end{equation}
\end{Prop}

The Voorhoeve index is very useful for counting the number of complex zeros of analytic functions.

        \subsubsection{Integral Frenet curvatures and spatial meandering}

\begin{Small}
Rotation of a smooth curve around a point outside it can be easily generalized for affine subspaces of higher dimensions. Let $\gamma$ be a smooth curve avoiding an affine subspace $A\subseteq\R^n$ in the Euclidean space.
Consider the orthogonal projection $\pi\:\R^n\to A^\perp$ which takes $A$ into a point $a$ in an affine subspace $A^\perp$ of complementary dimension. Then we can define rotation of $\pi\circ\gamma$ around $a$ inside $A^\perp$ as before, and use this nonnegative number as the measure of rotation of $\gamma$ around $A$. This construction works well until $\dim A\le n-2$.

Recall that any smooth Euclidean curve admits the osculating orthonormal frame defined outside of an exceptional (and generically small) number of points. If the initial curve is parametrized by a vector-function $t\mapsto x(t)=\bigl(x_1(t),\dots,x_n(t)\bigr)$, then the iterated derivations $v^{(1)}(t)=\frac{\d}{\d t}x(t)$, $v^{(2)}(t)=\frac{\d}{\d t}v^{(1)}(t)$, \dots, $v^{(n)}(t)=\frac{\d}{\d t}v^{(n-1)}(t)$ generically (i.e., for a generic curve and at a generic point $t$) define a frame in $\R^n$. This frame can be subjected to orthogonalization, producing vector-functions $\mathbf e_1(t),\dots,\mathbf e_n(t)$ such that
\begin{itemize}
 \item Vectors $v^{(1)}(t),\dots, v^{(k)}(t)$ span the same subspace as $\mathbf e_1(t),\dots,\mathbf e_k(t)$ for all $k=1,\dots,n$,
 \item The vectors $\mathbf e_1(t),\dots,\mathbf e_k(t)$ form an orthonormal tuple, positively oriented for $k=n$.
\end{itemize}
The collection $\mathbf e_1(t),\dots,\mathbf e_n(t)$ (positively oriented) is called the \emph{Frenet frame} associated with the curve $\gamma$. If the parametrization of the curve is \emph{natural} (by arclength), that is, $v^{(1)}=\mathbf e^1$, then one must have the \emph{Frenet identites}:
\begin{equation*}
  \frac{\d}{\d t}\begin{pmatrix}\mathbf e_1\\\mathbf e_2\\\vdots\\\mathbf e_{n-1}\\\mathbf e_n\end{pmatrix}=
  \begin{pmatrix}
  0&\varkappa_1\\
  -\varkappa_1&0&\varkappa_2\\
  &-\varkappa_2&0&\varkappa_3\\
  &&\cdots&0&\cdots\\
  &&&-\varkappa_{n-2}&0&\varkappa_{n-1}\\
  &&&&-\varkappa_{n-1}&0
\end{pmatrix}
  \begin{pmatrix}\mathbf e_1\\\mathbf e_2\\\vdots\\\mathbf e_{n-1}\\\mathbf e_n\end{pmatrix}
\end{equation*}
The quantities $\varkappa_i=\varkappa_i(t)$ are called the Frenet (generalized) curvatures ($i=1$ is the curvature, $i=2$ corresponds to torsion e.a.). For a generic smooth curve the curvatures $\varkappa_1,\dots, \varkappa_{n-2}$ are positive (nonvanishing), while the last curvature $\varkappa_{n-1}(t)$ may change sign but only at isolated points, called \emph{hyperinflections}. Note that in Theorem~\ref{thm:rolle-n} for a smooth curve $\gamma$ parametrized by the arclength, $\|S\gamma'\|=K_1(\gamma)$.

Denote by $K_i(\gamma)$ the absolute integral Frenet curvatures, $K_i(\gamma)=\int_0^\ell|\varkappa_i(t)|\,\d t$ (recall that the parametrization is by the arclength and for $\le n-2$ the curvatures are positive).

\begin{Thm}[D. Novikov, D. Nadler, S. Yakovenko \cites{ny-95,NaY}]\label{thm:ny-95}
Rotation of real analytic curve $\gamma$ around any $k$-dimensional affine subspace $A^k$ does not exceed $\pi(k+1)+4\sum_{i=0}^{k+1} K_i(\gamma)$. For closed curves the term $\pi(k+1)$ can be dropped, and rotation is bounded (up to a factor of 4) by the sum of integral absolute curvatures.
\end{Thm}
This result for $k=0$ differs from Theorem~\ref{thm:rolle-n} only by the factor 4 (and in this specific case it can be removed by a more detailed inspection). The Theorem can be extended for the case $\dim A=n-1$ as follows. Define ``rotation'' of a curve $\gamma$ around an $(n-1)$-dimensional subspace $A$ (affine hyperplane) as $\pi\cdot\#\{\gamma\cap A\}$. This definition can be justified if we consider the orthogonal projection of $\R^n$ on one-dimensional subspace $A^\perp$. The zero-dimensional ``unit sphere'' $\mathbb S_1^0$ consists of two points at distance $2$ from each other, but if we declare the ``spherical'' distance between these two antipodal points to be $\pi$ as for all higher dimensions, then this normalization becomes natural. In the same way it is natural to define the $n$-th integral curvature as $K_n(\gamma)=\pi\cdot\{\varkappa_{n-1}(t)=0\}$, the normalized number of the hyperinflection points. With these conventions, the inequality of Theorem~\ref{thm:ny-95} remains valid. We can restate it in the form not involving rotations in the extremal dimensions as follow,
$$
 \#\{\gamma\cap \varPi\}\le n +\frac4\pi \sum_{i=1}^{n-1}K_i(\gamma)\ +\ \#\{\varkappa_{n-1}=0\},
$$
where the first term can be dropped if the curve is closed.
\par\end{Small}

        \subsubsection{Non-oscillating curves in $\R^n$}\begin{Small}
Any curve in $\R^n$ can be cut by a suitable affine hyperplane at any $n$ points, which generically will be isolated. Curves which cannot be cut at more points, are called non-oscillating. Theorem~\ref{thm:ny-95} suggests that sufficiently small pieces of smooth curves are indeed non-oscillating. The problem is to make this claim qualitative as in Proposition~\ref{prop:dlVP}.

\begin{Def}
A spatial smooth curve is \emph{hyperconvex}, if it has no hyperinflection points, that is, the last Frenet curvature does not change its sign.
\end{Def}

\begin{Thm}[B. Shapiro, \cite{shapiro90}]
A hyperconvex curve in $\R^n$ such that
$$
\int_\gamma\sqrt{\varkappa_1^2(t)+\cdots+\varkappa_n^2(t)}\,\d t<\frac1{n\sqrt2}
$$
is non-oscillating.
\end{Thm}
This result can be compared to a corollary to Theorem~\ref{thm:ny-95} which for hyperconvex curves guarantees non-oscillation if $$\int_\gamma |\varkappa_1(t)|+\cdots+|\varkappa_{n-1}(t)|\,\d t<\frac\pi4.$$
\par\end{Small}

\subsection{Spatial curves vs.~linear ordinary differential equations}
The obvious parallelism between oscillation theory for linear ordinary differential equations and that for spatial curves is very easy to explain. A spatial curve is given by an $n$-tuple of smooth functions $x_1(t),\dots,x_n(t)$, and its intersections with an affine hyperplane are roots of (non-homogeneous) linear combinations $\sum_1^n c_ix_i(t)=c_0$. By the Rolle theorem, it is sufficient to estimate the number of roots of all homogeneous linear combinations $\sum_i c_i f_i(t)=0$, $f_i(t)=\dot x_i(t)$, which together satisfy a linear ordinary differential equation.

Any such equation can be written in the expanded form \eqref{lode}, where the coefficients $a_i(t)$ are obtained from the fundamental system of solutions $f_1,\dots,f_n$ by arithmetic operations and differentiation. However, sometimes division is to be used, hence explicit bounds for $\max_t |a_i(t)|$ in the spirit of the de la Vall\'ee Poussin theorem are problematic to establish.

The alternative is to reconstruct the differential operator vanishing on the given fundamental system of solutions in the form of a composition of alternating derivations and multiplications by functions with zeros and poles. Assume for simplicity that the functions $f_1(t),\dots,f_n(t)$ are real analytic, simply to avoid non-isolated intersections/zeros.

Denote by $W_k(t)$ the Wronskians of the first $k$ functions,
$$
W_k(t)=\det\begin{pmatrix}
f_1&f_2&\dots&f_k\\
f_1^{(1)}&f_2^{(1)}&\dots&f_k^{(1)}\\
\vdots&\vdots&\ddots&\vdots\\
f_1^{(k-1)}&f_2^{(n-1)}&\dots&f_k^{(k-1)}
\end{pmatrix}.
$$
These determinants can be considered as multilinear ordinary differential operators $\mathscr W(f_1,\dots,f_k)$ of order $k-1$ respectively (having different number of arguments), applied to the first $k$ functions in the tuple $f_1,\dots,f_n$.

Denote by $D_k$ the first order differential operators (written in the ``multiplicative'' form as composition of two 0-order multiplications by functions with the derivation squeezed between them)
$$
D_k=\frac{W_{k}}{W_{k-1}k}\cdot\frac\d{\d t}\cdot\frac{W_{k-1}}{W_k},
$$
which are in a sense derivations conjugated to the standard derivation by the operator of multiplication by the fraction $\frac{W_{k-1}}{W_k}$ (we agree that $W_0\equiv1$).

\begin{Lem}[G. P\'olya \cite{polya}]\label{lem:polya}
The functions $f_1,\dots,f_n$ satisfy the differential equation of order $n$ given by the composition
$$
D_nD_{n-1}\cdots D_2D_1y=0.
$$
\end{Lem}
If all Wronskians are non-vanishing, the system of functions $\{f_i\}_1^n$ is non-oscillating (Chebyshev), this follows from the classical Rolle theorem, as multiplication by a non-vanishing function does not change the number of zeros. When zeros of the Wronskians are allowed, one has to use the triangle inequality for the Voorhoeve index. This increases the bound, but the result will be almost sharp.

\begin{Small}
\begin{Rem}
Writing explicitly a linear $n$th order differential operator through its known $n$ linear independent solutions $f_1(t),\dots,f_n(t)$ is an easy task. The Riemann's solution is to write down the Wronski matrix of $n+1$ functions $f_1,\dots,f_n,y$ with an unknown function $y$,
$$
Ly=0,\qquad L=\mathscr W(f_1,\dots,f_n,y)
$$
and expand it in in the elements of the last column. Equating this sum to zero becomes a linear ordinary differential equation involving $y,y',\dots,y^{(n)}$ with coefficients being minoris of the Wronski matrix. It is well known that for analytic functions the equation holds if and only if $y, f_,\dots,f_n$ are linear dependent, that is, $y$ is a linear combination of $f_1,\dots,f_n$.

The Riemann formula gives the corresponding equation in the expanded form but is obviously independent of the ordering of the tuple $f_1,\dots,n$. On the contrary, the differential operator constructed in Lemma~\ref{lem:polya} is a composition of first order operators (a non-commutative analog of the representation of a polynomial as a product of linear forms corresponding to the roots of this polynomial). Checking the leading coefficients shows that $L=D_nD_{n-1}\cdots D_1$. Not surprisingly, such ``multiplicative'' representation makes it much easier to solve the equation.

Of course, the P\'olya form can be transformed to the Riemann form by repeated application of the Leibniz rule (leading to a non-commutative version of the Vieta formulas). But in general, the P\'olya form (the coefficients of the operators $D_k$) \emph{depends explicitly} on the ordering of the tuple, since decomposition is not unique. Consider, say, the example of the functions $1,t,t^2,\ldots,t^{n-1}$ which will produce the obvious Riemann equation $y^{(n)}=0$, and compare it with a P\'olya form computed for a permuted tuple.
\end{Rem}

\begin{proof}[The idea of the proof of Theorem~\ref{thm:ny-95}]
Consider the $n$-space curve $\G=\G_n$ parameterized by its coordinate functions $x_i=f_i(t)$. Its osculating frame is formed by the vectors $\mathbf v_1=\dot x,\mathbf v_2=\ddot x,\cdots,\mathbf v_n=\tfrac{\d^n}{\d t^{n-1}}x$. Orthogonalization of this frame is the Frenet frame $\mathbf e_1(t),\dots,\mathbf e_n(t)$ ruled by the Frenet equations above. Note that the vectors $\mathbf v_i$ are naturally ordered by the order of the respective derivative. Denote by $V_k$ the Gram--Schmidt determinant of the first $k$ vectors, $V_k(t)=\det^{1/2}\bigl\|\<\mathbf v_i(t),\mathbf v_j(t)\>\bigr\|_{i,j=1}^k$. It is easy to check that
$$
\varkappa_k(t)=\frac{V_{k-1}(t)V_{k+1}(t)}{V_k^2(t) V_1(t)}, \qquad k=1,\dots,n-1\text{ (for $k=1$ we assume }V_0\equiv1).
$$
Application of the P\'olya formula would yield the number of intersections in terms of the number of zeros of $\varkappa_1,\dots,\varkappa_{n-1}$ on $[0,t]$, which is unknown. But one can apply the Buffon needle principle, which expresses the average number of intersections of a curve in $k$-space with a ``random'' hyperplane through the length of the spherical projection of this curve on the sphere $\mathbb S^{k-1}\subset\R^k$. This allows to place upper bounds for the number of \emph{average} intersections with projections $\G_1(F), \dots, \G_{n-1}(F)$ of the curve $\G_n$ on subspaces of a \emph{``random''} complete flag $F=\{L_1\subset L_2\subset \cdots\subset  L_{n-1}\subset\R^n\}$ as linear combinations of the corresponding curvatures. The ultimate case $\G_1$ is nothing but the classical Rolle lemma (the number of zeros of $\varkappa_{n-1}$ is bounded in terms of the roots of its derivative).

It only remains to notice that one always find a flag in $\R^n$ for which the numbers of intersections would be no greater than their values averaged over all flags. This flag should replace the flag generated by the initial coordinates $x_1,\dots,x_n$ by their linear combinations $\~x_1,\dots,\~x_n$ so that the respective curvatures $\~\varkappa_k(t)$ will have the  number of zeros not exceeding the sum of the  initial integral curvatures $K_k$.
\end{proof}
\par\end{Small}

\section{Counting complex roots}

\subsection{Kim theorem}
Yet the simplest, the most direct generalization of Proposition~\ref{prop:dlVP} to complex settings can be obtained by a simple modification of the real proof.

Assume that $U$ is a convex domain of diameter $\ell>0$ and $f$ a function holomorphic in $U$ and continuous in $\overline U$ (for simplicity we will assume that $\partial U$ is piecewise smooth).

\begin{Lem}\label{lem:PreKim}
If $f$ has $n$ isolated zeros in $U$, then for all $k=0,1,\dots,n-1$ the inequality $\|f^{(n-k)}\|\le \|f^{(n)}\|\cdot \frac{\ell^{k}}{k!}$ holds.
\end{Lem}

\begin{Small}
\begin{proof}[The idea of the proof]
Knowing the $n$th derivative allows to restore a function uniquely from its zeros $a_1,\dots,a_n$ by iterated integration, using the formal operator identity
$$
\partial^n=\bigl((z-a_n)\partial+n\bigr)\cdots\bigl((z-a_1)\partial+1\bigr).
\frac1{(z-a_n)\cdots(z-a_1)}
$$
\end{proof}
\par\end{Small}

This Lemma immediately implies (exactly as in the real case) the complex non-oscillation result.

\begin{Thm}[W. J. Kim \cite{kim}]
Consider a linear ordinary differential equation of the form
\begin{equation}\label{hlode}
  y^{(n)}+a_1(t)\,y^{(n-1)}+\cdots+a_{n-1}(t)\,y'+a_n(t)\,y=0, \qquad t\in U\Subset\C
\end{equation}
in a bounded convex complex domain $U$ of diameter $\ell$ with the coefficients $a_i(z)$ holomorphic in $U$ and bounded there, $\|a_k(z)\|\le A_k<\infty$, cf.~with the real counterpart \eqref{lode}.

If the domain $U$ is small enough so that the inequality \eqref{lode-bounds} holds, then any solution of the equation \eqref{hlode} has at most $n-1$ isolated zeros in $U$.\qed
\end{Thm}

As in the real case, any larger domain can be subdivided into smaller domains, yielding explicit bounds for the number of zeros of solutions for equations with bounded coefficients. However, appearance of singular points makes such subdivision very problematic and alternative tools are required.

\subsection{Jensen inequality}\label{sec:jensen}
The argument principle for holomorphic functions implies that a function that has many zeros in a domain, should have a very fast rotating argument on the boundary. Since (outside roots) argument and logarithm of modulus are harmonically conjugate to each other, one should expect a similar connection between the growth rate of analytic functions and its number of zeros.

In its simplest form this principle is reflected in the Jensen formula. Assume that a function $f$ is holomorphic on the closed unit disk $\overline{\mathbb D}$, i.e., $f\in\O(\overline{\D})$, in particular, continuous on its boundary $\mathbb S=\partial\mathbb D$, and $f(0)\ne 0$. Then
\begin{equation}\label{jensen}
  \ln |f(0)|=\sum\ln |z_i|+\frac1{2\pi}\int_{\mathbb S_1}\ln |f(z)|\,|\d z|,
\end{equation}
where the summation is extended over all isolated roots of $f$ in $\mathbb D$.
If we denote $m=|f(0)|$, $M=\max_{\mathbb S}|f(z)|$, then this implies the inequality
\begin{equation}\label{jensen-in}
  \sum_{f(z_i)=0} -\ln |z_i|\le \ln \frac Mm,
\end{equation}
that is, the total count of roots $z_i$ of $f$ in $\mathbb D=\{|z|<1\}$ with \emph{positive} weights $-\ln |z_i|$  is bounded by the growth rate of $f$ from the center to the boundary of $\mathbb D$. This allows to count the (unweighted) number of zeros of $f$ in any smaller disk $(1-\e)\mathbb D$, $0<\e<1$. Alternatively, this inequality becomes the cornerstone of the Nevanlinna theory which connects distribution of zeros of entire functions with their growth as measured by the maximum modulus on expanding concentric disks $r\mathbb D$. In what follows for a function $f$ analytic on a closure of a  bounded open set $U\subset \C$ and $K\Subset U$ we will denote
\begin{gather*}
M_K(f)=\max_{z\in K}|(f)|,\qquad\text{and}\quad M_U(f)=M_{\overline U}(f). 
\end{gather*}
Then \eqref{jensen-in} implies
\begin{gather*}
N_{\frac12R\mathbb D}(f)\le \ln 2\cdot \ln \frac{M_{R\mathbb D}(f)}{M_{\frac12R\mathbb D}(f)}.
\end{gather*}
However, for functions with singularities we will need domains of more complicated shapes.

\subsection{Bernstein index}
From now on (and mainly for simplicity) we will consider only relatively tame subsets of the complex plane, namely, curvilinear polygons (CP-gons), bounded simply connected open domains with piecewise analytic boundaries. If $\overline K\Subset U$, then the \emph{gap} between two CP-gons $K$ and $U$, defined as
\begin{equation}\label{gap}
\gap(K,U)=\max_{\e>0}\{\e: K+\e \mathbb D\subseteq U\},
\end{equation}
is positive. Besides, for such nested pairs (and triples) we will consider only functions holomorphic on $\overline U$, avoiding thus any potential troubles with the boundary behavior.

\subsubsection{On the order of quantifiers: how to understand the inequalities below}
In the Nevanlinna theory the usual setting is as follows: an entire function $f$ is given and then one studies distribution of its isolated roots in large disks (when their radius $R$ tends to infinity), making a stress on asymptotic dependence on $R$ whereas the constants are allowed to depend on $f$.

Below we partially invert the settings. We consider pairs of properly nested CP-gons $K\Subset U$ and for each pair define one or several functionals on the space of functions holomorphic in $U$ and continuous in $\overline U$. The maximum modulus functions $M_K(U)$ and $M_U(f)$ are typical examples: they are \emph{first order homogeneous} (nonlinear) functionals, while their ratio $M_R(f)/M_K(f)$ will be zero order homogeneous (the same for $f$ and $\l f$ for any $\l\ne 0$).

The results below purport to give pointwise bounds for $|f(z)|$ in terms of these functionals, valid on all subsets or on substantial parts thereof.

\subsubsection{Definition of the Bernstein index}
By the Riemann uniformization theorem, if $K\Subset U\subsetneq\C$,
then a simply connected domain $U$ can be equipped with the hyperbolic metric,
and the diameter of $K$ in this metric is finite. The Jensen inequality implies then the following statement.

\begin{Thm}
There exists a positive finite constant $\beta=\beta_{K,U}$ depending only on the hyperbolic diameter of $K$ in $U$ with the following property.

For any function $f\not\equiv0$ holomorphic in $U$ and continuous in $\overline U$ the number of its \textup(necessarily isolated\textup) zeros in $K$ is bounded by $\beta \ln\frac{M_U(f)}{M_K(f)}\ge0$.
\end{Thm}

\begin{Def}
For a pair of sets $K\Subset U$ as above, the Bernstein index of a function is the ratio
\begin{equation}\label{bernstein-sets}
B=B_{K,U}(f)=\ln\frac{M_{U}(f)}{M_K (f)}\ge 0.
\end{equation}
\end{Def}
The maximum modulus principle implies that $B(f)=0$ if and only if $f$ is a (nonzero) constant.

\begin{Rem}
The term ``Bernstein index'' was suggested in \cite{iy} because of the classical Bernstein inequality: if $K=[-1,1]$ and $U_R=\{|z-1|+|z+1|\le 2R\}$ is the ellipse with foci at $\pm 1$ of ``radius'' $R>0$, then for any polynomial $p$ of degree $n$ the inequality $B_{K,U_R}(p)\le n\ln R$ holds, the bound being sharp.
\end{Rem}

The constant $\beta_{K,U}$ is explicit and can be easily bounded from above in simple cases. The Jensen formula provides an explicit bound for a pair of concentric disks: in particular, if $U=\mathbb D$ is the unit disk and $K_\e$ the concentric disk of radius $1-\e<1$ (so that the gap is $\e$), then $\beta_{K_\e,U}\sim O(1/\e^2)$. If $K$ is a real segment of length $|K|$ and $U_\e=K+\e\mathbb D$ its complex $\e$-neighborhood (``stadium''), then $\beta_{K,U_\e}\sim \exp O(|K|/\e)$, see \cite{iy}.

In what follows we consider the Bernstein index as a (nonlinear) functional on nonzero analytic bounded functions with suitable domain of definition. This functional, denoted $B_{K,U}$, is zero-order homogeneous and depends monotonously on $K$ and $U$: if $K'\subseteq K\Subset U\subseteq U'$ (recall that we deal only with tame CP-gons), then by the maximum modulus principle
\begin{equation}\label{bern-monot}
  B_{K',U}\ge B_{K,U}\ge B_{K,U'}.
\end{equation}

\begin{Small}
\begin{Rem}
The inequalities discussed in this section are not very difficult to prove using two key results. One is the Cartan inequality asserting that a \emph{monic} polynomial $p(z)=z^n+a_1z^{n-1}+\cdots+a_n\in\C[z]$ in one complex variable cannot be uniformly small away from its zero locus (on a controlled distance from its zero locus), see \cites{cartan,lub}. The other ingredient is the Harnack inequality which gives two-sided upper and lower bounds for a function $u$, \emph{harmonic and positive} in the open unit disk $U=\D$, in terms of the distance to the boundary circle, see \cite{evans}:
$$
\frac{1-|z|}{1+|z|}u(0)\le u(z)\le \frac{1+|z|}{1-|z|}u(0),\qquad |z|<1.
$$
This equation applied to the absolute value $\log |f(z)|$ of a holomorphic function having no zeros in the unit disk, gives an explicit lower bound for $\max_{z\in K}|f(z)|$ for any compact $K\Subset U$.

Combination of these two ingredients with the Jensen inequality for the number of isolated zeros from \secref{sec:jensen} allows to give a lower bound for $M_K(f)$ for all $K\Subset U$ and all holomorphic $f\in\O(U)$.
\end{Rem}\par
\end{Small}

\subsection{Variation of argument of solutions of complex-valued linear equations}
Consider the linear equation \eqref{lode}, but assume this time that its coefficients $a_j(t)$ are \emph{complex}-valued (say, continuous) functions of the real variable $t\in[0,\ell]\Subset\R$. Then nontrivial solutions of this equation will be in general also complex-valued and will not have zeros at all.

Yet instead one can count the Voorhove index of these solutions. It turns out that the following immediate analog of the de la Vall\'ee--Poussin theorem holds, see \cite{fields}*{Theorem 2.6}.

Assume, as before, that $|a_k(t)|\le A_k<+\infty$.

\begin{Thm}[S. Yakovenko \cite{fields}, based on an idea by G. Petrov \cite{petrov}]\label{thm:lode-voor}
Assume that the complex coefficients of the equation \eqref{lode} are so small that
\begin{equation}\label{hlode-bounds}
  \sum_{i=1}^n A_k\frac{\ell^k}{k!}<\frac12
\end{equation}
cf.~with \eqref{lode-bounds}. Then for any solution $f$ of the equation \eqref{lode} we have the inequality on the variation of argument along $[0,t]$
$$
|\Arg f(\ell)-\Arg f(0)| < (n+1)\pi.
$$
\end{Thm}

The main idea of the proof is the following observation which is very much in the spirit of the elementary Rolle lemma.

\begin{Prop}\label{prop:Petrov-Rolle}
If any continuous complex-valued function $f$ makes more than a half-turn around the origin on some connected interval, then both imaginary and real part of this solution must have a zero on this interval.
\end{Prop}

Then the inequalities from Lemma~\ref{lem:symplex} can be repeated verbatim: if the rotation of a solution is more than $n+1$ half-turns, then both real and complex parts of $f$ must have at least $n$ zeros and the \eqref{hlode-bounds} is impossible. \qed

\subsection{Rolle and triangle inequalities for the Bernstein index}
Let $K'\Subset K\Subset U$ be a triple of CP-gons with $\gap(K',K)>0$.
\begin{Thm}
There exists a finite positive constant $\rho=\rho(K',K,U)$ such that for any function $f$ analytic in $U$ and its derivative $f'$,
\begin{equation}\label{rolle-bern}
  B(f)\le B'(f')+\rho, \qquad \rho=\rho(K',K,U)<+\infty.
\end{equation}
The ``Rolle defect'' $\rho$ can be explicitly computed in terms of gaps between the three sets.
\end{Thm}

\begin{Small}
\begin{proof}[The idea of the proof]
Consider the intermediate contour $\G=\partial K$ which encircles $K'$ and is inside $U$, and take an arbitrary function $f$ normalized so that $M_K(f)=1$. Then by the Cauchy estimate we have an explicit \emph{upper} bound for $M_{K'}(f')$.

On the other hand, the growth of $f$ between $K$ and $\partial U$ is bounded by the maximal value of the gradient of $f$, that is, in terms of $M_U(f')$. This implies a lower bound for $M_U(f')$ in terms of $B_{K,U}(f)$. Together these two inequalities prove the Theorem, accurate details can be found in \cite{ny-97}.
\end{proof}\par\end{Small}

\subsubsection{Application to pseudopolynomials}\label{sec:pseudopol}
This allows to apply Theorem~\ref{thm:rolle-c} to the study of roots of \emph{pseudopolynomials},  functions of the form
$$
 p(z)=\sum_{\l\in\L}\mathrm e^{\l z} p_\l(z), \qquad \L\subset \C,\quad p_\l\in\C[z],
$$
where $\L$ is a point set called \emph{spectrum} of $p$. If $\deg p_\l= n_\l>0$, then we say that $\l$ is a point of multiplicity $n_\l$ in the spectrum. The total number $|\L|$ of points in $\L$ counted with their multiplicities is called the degree of a pseudopolynomial.

For each pseudopolynomial of degree $n$ one can construct by induction a composition of derivations and multiplications by suitable exponential functions $\mathrm e^{\mu z}$, depending on $\L$, which takes $p$ into a nonzero constant. Such combination will involve at most $n-1$ derivations $\frac{\d}{\d z}$. Voorhoeve index of an exponential function $\mathrm e^{\mu z}$ admits an upper bound in terms of $|z|$ and the diameter of $U$, which by the triangle inequality yields an upper bound for $V_{\partial U}(p)$ for any bounded convex domain $U$.

\subsection{Bernstein index for power series}
There is one more reincarnation for the Bernstein index. Let $\D\subseteq\C$ be the unit disk. Then any function $f\in\O(\overline{\D})$ holomorphic in a neighborhood of $\overline{\D}$ can be expanded in the Taylor series
\begin{equation}\label{taylor}
f(z)=\sum_{k=0}^\infty a_k z^,\qquad |z|<1, \quad \{a_k\}_{k=0}^\infty\subseteq\C.
\end{equation}
This series converges absolutely for all $|z|\le1$.

The sequence of coefficients $\{a_k\}$ could be thought of as a function
$$
a\:\Z_+\to\C,\qquad k\mapsto a_k,
$$
and referred to as the Cauchy transform $\mathscr Cf$ of $f$. Compact subsets of $\Z_+$ are finite intervals $0\le k\le \nu$, and for any such $\nu\in\N$ one can compare the ration
\begin{equation}\label{bern-yom}
\frac{\sup_{k\in\Z_+}|a_k|}{\max_{k\le \nu}|a_k|}\le+\infty,
\end{equation}
cf.~with \eqref{bernstein-sets}, where $a=\mathscr Cf$ and the role of $U$ is played by $\Z_+$.

\begin{Def}[Bernstein classes after Roytwarf--Yomdin \cite{RoY}]\label{def:bern-yom}
Let $c\in\R_+$ and $\nu< +\infty$. The (second) Bernstein class with parameters $c,\nu$ is the collection of functions $f\in\O(\D)$ such that
$$
\frac{\sup_{k\in\Z_+}|a_k|}{\max_{k\le \nu}|a_k|}\le c.
$$
\end{Def}
Speaking informally, functions from the second Bernstein class can be considered as perturbations of polynomials of degree $\nu$: when the first $\nu$ Taylor coefficients vanish, the function must vanish identically.

\begin{Rem}
One can easily modify the above construction for functions defined in  a neighborhood of the closure of the disk $R\cdot\D$, $R>0$. Then all absolute values of the coefficients should be replaced by $|a_k|R^k$.
\end{Rem}

\begin{Thm}[Equivalence theorem, \cite{RoY}]\label{thm:equivalence}
Let $f$ be a function from the second Bernstein class with parameters $c,\nu$ and $R=1$.

Then for any $\alpha<\beta<1$ the Bernstein index \eqref{bernstein-sets} $B_{K,U}(f)$ for $U=\beta\D$ and $K=\overline{\alpha\D}$ can be explicitly bounded in terms of $c,\nu$.

Conversely, a function whose Bernstein index $B_{\overline{\alpha\D},\beta\D}(f)$ is finite, belongs to the second Bernstein class with parameters $c,\nu$ explicitly bounded in terms of $\alpha,\beta$ and the Bernstein index. \qed
\end{Thm}

Again, not surprisingly, for functions from the second Bernstein class one can explicitly bound the number of isolated roots.

\begin{Prop}
For a function $f$ from the second Bernstein class with parameters $c,\nu$ and any $0<\alpha<1$ one can place an explicit bound for the number of isolated roots of $f$ in $\alpha\D$.

Reciprocally, one can give an explicit \emph{lower bound} for the radius $\rho >0$ \textup(in terms of $c$\textup) such that the number of isolated zeros of $f$ in the disk $\rho\D$ does not exceed $\nu$. \qed
\end{Prop}

One should remark here that the ``Bernstein inequality for the Cauchy transform'' as it appears in the Definition~\ref{def:bern-yom} is a very convenient tool for studying zeros of polynomial differential equations whose Taylor coefficients satisfy recursive equations. We will not go into details, referring instead to the works of Y. Yomdin, J.-P. Fran\c coise, M.~Briskin, N.~Roytwarf e.a. An alternative proof of the Equivalence theorem~\ref{thm:equivalence} is given in \cite{nonlin-00}.

\subsection{Singular points and Rolle theory for difference operators}
In this section we briefly explain how zero counting techniques exposed above can be generalized for solutions of homogeneous linear ordinary equations near singular points. The complete exposition can be found in \cite{fields}.
\subsubsection{Fuchsian singularities}
A \emph{monic} differential equation of the form \eqref{hlode} with coefficients $a_j(t)$, $j=1,\dots,n$ \emph{meromorphic} in $U\Subset\C$ has singular points (singularities) where at least one of the coefficients $a_j$ has a pole. In general, solutions of the equation are ramified (multivalued) at singularities, but besides ramification, the nature of solutions depends very much on the orders of the poles of the coefficients $a_1,\dots,a_n$. In the simpler case, called \emph{Fuchsian}, or regular, solutions exhibit properties similar to those of poles of finite order, like the functions $t^\l$ at the origin. If the complementary (wild) case solutions behave like essential singularities of the form $t^\l\exp (1/t)$. For the reasons similar to that of Picard theory, in the wild case there is no hope for any meaningful zero counting, as almost every value is assumed by any solution infinitely many times in any (punctured) neighborhood of the singularity.

On the contrary, the Fuchsian case admits explicit solution in terms very close to the nonsingular case.

To write down explicitly the Fuchs condition, note that a \emph{homogeneous} linear differential equation can be (i) re-expanded in terms of any differential operator $v(t)\frac{\d}{\d t}$ with a meromorphic function $v$ and (ii) multiplied by any meromorphic function. Assuming the singularity at the point $t=0\in U$, we introduce the Euler operator $\eu=t\frac{\d}{\d t}$ and then multiply the initial equation \eqref{hlode} by a suitable power of $t$ in such a way that all coefficients become \emph{holomorphic} but not vanish at $t=0$ simultaneously. The resulting equation will have the form
\begin{equation}\label{fode}
Ly=0,\qquad L=b_0(t)\eu^n+b_1(t)\eu^{n-1}+\cdots+b_n(t).
\end{equation}
\begin{Def}
The point $t=0$ is Fuchsian if $b_0(0)\ne0$, that is, if the operator $L$ above can be assumed \emph{monic with holomorphic coefficients}, $L=\eu^n+\sum_{j=1}^n b_j\eu^{n-j}$ (``$\eu$-nonsingular'').
\end{Def}

\begin{Ex}\label{ex:charnos}
Assume that all coefficients $b_j\in\C$ are (complex) constants. Then \eqref{fode} becomes the classical Euler equation. Its solutions are linear combinations of \emph{pseudomonomials} of the form $t^\l\ln^k t$ (cf.~with \secref{sec:pseudopol}), where $\l$ is a \emph{characteristic number}, the root of the characteristic polynomial $\l^n+\sum b_j\l^{n-j}$ and $k$ is an integer smaller than the multiplicity of this root, so that their total number is exactly $n$.
\end{Ex}

As was already mentioned, the Fuchsian condition is \emph{equivalent} to the property that all solutions of the equation \eqref{hlode} exhibit moderate growth as $t\to0$ while $\arg t$ remains bounded  (every branch of the multivalued solution grows no faster than $C|t|^{-N}$ for some finite $C,N$).

Further we assume that $0$ is the only singularity of $L$ in $\overline{\D}$ (this can be achieved by rescaling $t$).

Geometrically the definition of Fuchsian equations means that in the logarithmic chart $z=\ln t$ the equation assumes the form
\begin{equation}\label{flode-z}
y^{(n)}+\sum_{j=1}^n b_j(z)\,y^{(n-j)}=0,
\end{equation}
where the derivatives are taken with respect to the new independent variable $z$ as above, and the coefficients $b_j(z)$ are $2\pi\mathrm i$-periodic and have limits as $\Re z\to-\infty$.

In this situation the zero counting problem in the (slit) unit disk reduces to counting zeros in the semi-\emph{infinite} semistrip $\Pi\subseteq\C$ bounded by two horizontal lines $\Im z=\pm2\pi$ and the vertical segment $\Re z=0$.

\subsubsection{Argument principle for unbounded domains: Petrov difference operators and the associated Rolle theory}
The semi-infinite strip $\Pi$ can be cut off to a finite rectangle $\Pi_c$ bounded from the left by a vertical segment $\Re z=-c$, $c\gg 0$: the bound for this rectangle will serve automatically as the bound for $\Pi$ if uniform in (independent of) $c$.

The argument principle applied to $\Pi_c$ would yield such a bound if explicit upper bounds for the variation of argument of any solution would be available for all four sides of the rectangle. For the two vertical sides it is indeed possible by virtue of Theorem~\ref{thm:lode-voor}. Indeed, the lengths of the two segments are bounded by $4\pi$, while the magnitude of coefficients of the equation \eqref{flode-z} are bounded in the left half-plane $\Re z\le 0$ uniformly by some explicit constant. Denote by $B$ the (finite and explicit) bound
\begin{equation}\label{bound-vert}
B=\sup |\Arg f(z + 2\pi\mathrm i)-\Arg f(z - 2\pi\mathrm i)\,|
\end{equation}
where the supremum is taken over all solutions of the equation \eqref{flode-z} and all vertical segments $[z-2\pi\mathrm i, z+2\pi\mathrm i]$ of length $4\pi$ in the left half-plane.

What constitutes the genuine problem is the variation of argument along the two long horizontal segments
\begin{equation}\label{var-horiz}
|\Arg f(\pm 2\pi\mathrm i)-\Arg f(-c \pm 2\pi\mathrm i)\,|, \qquad c\gg 1.
\end{equation}
Since the length $c$ of these segments is unbounded,  Theorem~\ref{thm:lode-voor} does not apply.

Denote by $\Delta$ the operator defined on functions holomorphic in the left half-plane $\mathbb H=\{\Re z\le 0\}$ as the argument shift by $2\pi\mathrm i$ by the formula
$$
(\Delta f)(z)=f(z+ 2\pi\mathrm i).
$$
This operator preserves the coefficients of the equation \eqref{flode-z} but is a nontrivial isomorphism on the set of solutions of this equation, called the monodromy operator. Using $\Delta$, for any complex number $\mu\in \C$ one can define the \emph{difference operator}
$$
P_\mu=\mu^{-1}\Delta-\mu\Delta^{-1}
$$
called the \emph{Petrov} operator in \cite{ry}, where it was introduced, see also \cite{fields}.

This operator satisfies the following propertiy.

\begin{Lem}[Rolle inequality for difference operators]\label{lem:petrov}
If $\mu$ is of modulus 1, i.e., $\bar \mu=\mu^{-1}$ and $f$ is a solution of \eqref{flode-z} holomorphic in  $\Pi$, real on $\R$ and with a finite explicit upper bound \eqref{bound-vert}. Then
\begin{equation}\label{rolle-diff}
N(f)\le (2B+1) +N(P_\mu f),
\end{equation}
where $N(f)$ and $N(P_\mu f)$ are the number of \emph{isolated} zeros of $f$ and $P_\mu f$ in $\Pi$ respectively.
\end{Lem}

\begin{proof}
Replace $\Pi$ by $\Pi_c$ and apply the argument principle to $f$.

The variation of argument along the boundary does not exceed $2B$ (contribution of the two vertical sides) plus the contribution along the two long horizontal segments \eqref{var-horiz}. But $\Delta^{-1}f=\overline{\Delta f}$ and $\bar \mu=\mu^{-1}$, therefore $P_\mu f$ on these sides equals to the imaginary part of $\mu f$ on these sides. The bound now follow from
$$
\#\{\Im f|_\gamma=0\}\ge \frac 1 \pi \Delta_\gamma \Arg f -1
$$
for any curve $\gamma(t)$, which is a direct corollary of Proposition~\ref{prop:Petrov-Rolle}.
\end{proof}

    \subsubsection{Zeros of Fuchsian singularities}
Lemma~\ref{lem:petrov} can be iterated. Indeed,  the monodromy operator $\Delta$ of a Fuchsian equation \eqref{fode} is an isomorphism of the complex linear $n$-space of its solutions with the eigenvalues $\mu_j=\exp 2\pi i \l_j$, where $\l_j$ are the characteristic numbers, see Example~\ref{ex:charnos}.  Dimension of the corresponding root space is less or equal to the multiplicity $\nu_j$ of the corresponding number.

Analytically this means that solutions of the equation \eqref{fode}, resp., \eqref{flode-z}, can be written as finite sums of the form
\begin{equation}\label{jordan}
f(t)=\sum_{j,k} t^{\mu_j}\ln^k t\cdot\f_{j,k}(t),\quad \text{resp.},\quad f(z)=\sum_{j,k}z^k\exp (\l_j z)\cdot \f_{j,k}(e^t),
\end{equation}
where $\f_{j,k}(t)$, $k\le \nu_j-1$, are holomorphic in $\overline{\D}$.
The composition of the commuting Petrov operators
$$
P=\prod_{j}P_j^{\nu_j}, \qquad P_j=P_{\mu_j}
$$
annuls all solutions of the equation, since each factor annuls the corresponding root space, $Pf\equiv0$ and $N(Pf)=0$. This immediately implies the following result.

\begin{Thm}\label{thm:roitman}
For any real Fuchsian equation \eqref{fode} in a  disk $\mathbb D$ free from other singularities and having only real characteristic numbers, the number of isolated zeros of solutions is explicitly bounded in terms of the relative magnitude of the non-principal coefficients
$$
M=\max_j\max_{t\in \mathbb D}\frac{|b_j(t)|}{|b_0(t)|}<+\infty.\qed
$$
\end{Thm}

\subsection{Pseudo-Abelian integrals}
Pseudo-Abelian integrals are    integrals of polynomial 1-forms over families of real ovals $\gamma_t\subset\{H=t\}$, where $H=H_1^{\lambda_1}\cdots\ H_n^{\lambda_n}$, $\lambda_i\in\R_+$,  is a Darboux-type first integral of a Darboux integrable  polynomial   vector field on $\R^2$ with suitable $H_i\in\R[x,y]$.

	The above arguments allow to provide a locally uniform bound on the number of pseudo-Abelian integrals. Let $\mathscr{D}=\mathscr{D}_{n, d_1,\ldots,d_n,d}$ be a finite-dimensional space
$$
\{\lambda=(\lambda_1,\ldots,\lambda_n, H_1,\ldots, H_n,\omega), \lambda_i\in \R_+, \deg H_i\le d_i, \deg\omega\le d\}
$$
of parameters  defining the pseudo-Abelian integral $I_\lambda(t)=\int_{\gamma_t}\omega$.
\begin{Thm}
For a generic $\lambda\in \mathscr{D}$ there exists some $\delta>0$ such that
	  the number of   zeros of $I_{\lambda'}(t)$ in $(0,\delta)$ is uniformly bounded by some constant $c=c(\lambda)$ for all $\lambda'\in\mathscr{D}$  sufficiently close to  $\lambda$.
\end{Thm}
Note  that the pseudo-Abelian integrals do not satisfy any linear ODE. However, the above approach works, and produces a non-effective but \emph{uniform} upper bound.
For details see \cites{NovPAI,BobMar}.

\emph{Sketch of the proof:} A pseudo-Abelian integral $I(t)$ has a branching point at the origin  $t=0$ corresponding to a reducible algebraic curve and admits a \emph{convergent} representation which depends analytically  on the parameters $\lambda_i$ and the coefficients of $H_i$:
\begin{equation}\label{eq:pseudoAI convergent}
I(t)=\sum_{i,j=1}^{n}\sum_{p,q\ge 0} a_{prij}\ell_{pqij}+\sum_{i=1}^n f_i(t^{1/\lambda_i}),
\end{equation}
where
\begin{equation}\label{roussarie}
	\ell_{pqij}(t)=
	\begin{cases}
		\dfrac{t^{p/\lambda_i}-t^{q/\lambda_j}}{p/\lambda_i-q/\lambda_j},&\text{if }p/\lambda_i\ne q/\lambda_j,
		\\[10pt]
		{t^{p/\lambda_i}\log t}&\text{otherwise}
	\end{cases}
\end{equation}
is a generalized \emph{Roussarie-Ecalle} compensator, cf.~with \cite{roussarie}.
Note that collecting similar terms of the above double sum into
\begin{equation*}
\sum_{i=1}^{n}\hat{f}_i(t^{1/\lambda_i})+\sum_{i=1}^{n}\hat{g}_i(t^{1/\lambda_i}\log t) ,
\end{equation*}
with result in an asymptotic series for $I(t)$ with generally \emph{divergent} formal Laurent power series $\hat{f}_i, \hat{g}_i$ due to unavoidable presence of \emph{small denominators} in \eqref{roussarie}.

Applying Lemma~\ref{lem:petrov} to $I(t^{\lambda_1})$, we reduce the problem of finding an upper bound on the number of zeros of $I(t)$ near the origin to that for the number of zeros of $PI(t^{\lambda_1})$, where $P=P_{1/\l_1}$ is the corresponding Petrov operator. The function $PI$ has a  representation similar to \eqref{eq:pseudoAI convergent} but with $n$ replaced by $n-1$ (actually, topological arguments of \cite{BobMar} show that $PI(t^{\lambda_1})$ is also a pseudo-Abelian integral of the same form but along a combinatorially simpler loop). Applying the suitable Petrov operators $n$ times, we get a locally uniform in parameters upper bound for the number of zeros of $I(t)$ near the origin for generic collections of $H_i$.

This result gives hopes for a generalization of Varchenko-Khovanskii finiteness theorem for pseudo-Abelian integrals (i.e. a uniform bound for the whole $\mathcal{D}$), and local boundedness was further proved for generic cases  of codimension one in $\mathscr{D}$, see \cites{BMN, BMN2}.

\section{Many (complex) dimensions}

When discussing the intersection theory in many (complex) dimensions, we need to distinguish three types of results which usually have the following form. There is a number $\G^1,\G^2,\dots$ of subvarieties in $\C^n$ or in an open domain $U\subset\C^n$ defined in various ways: they could be algebraic subvarieties, integral surfaces for systems of rational Pfaffian equations, common integral manifolds of commuting rational vector fields etc.; they may be assumed smooth or singularities can be allowed, and they come in the naturally parameterized families: e.g., a hypersurface $\G=\G_0$ defined by an equation $F(x)=0$ is naturally embedded in the family $\G_\e=\{F(x)=\e\}$ of the level sets $\e\in(\C,0)$. In a similar way, integral subvarieties are naturally parameterized by their intersection points with a suitable transversal manifold of complementary dimension, being leaves of the appropriate foliations. To simplify the speech, we will refer to them as \emph{nearby subvarieties}.

We usually assume that the codimensions of the varieties $\G^k$ add up to the number $n$, the dimension of the ambient space, so that generically the intersection $\bigcap\G^k$ will be zero-dimensional and consist of isolated points that could be counted.

This count can be restricted by the location of these intersections in the increasingly greater ``domains''.
\begin{enumerate}
 \item \emph{Infinitesimal flavor}: assuming that a certain point (e.g., the origin in $\C^n$) is an isolated intersection of the subvarieties $\G^k_0$, what can be the maximal multiplicity of this intersection at the given point? An explicit answer $\mu<+\infty$ means that for any sufficiently small neighborhood $U$ of the origin all nearby subvarieties $\G^k_{\e_k}$ would intersect by no more than $\mu$ isolated points in $U$ for all sufficiently small $\e_k$, $k=1,2,\dots$.
 \item The general theory asserts only the existence of such small neighborhood $U$. An \emph{explicit infinitesimal problem} is to give an \emph{explicit} lower bound for the radius $R>0$ such that the ball $\{|x|<R\}$ contains no more than $\mu$ isolated intersections of the nearby subvarieties.
 \item It may well happen that the original intersection $\bigcap\G^k_0$ is non-isolated (has multiplicity $\mu=+\infty$). Yet by the Sard theorem, the majority of perturbed subvarieties $\G^k_{\e_k}$ are expected to intersect by isolated points (at least for sufficiently generic parametric families). Thus one can formulate the \emph{local counting problem} for the number of isolated intersections of nearby subvarieties in a small enough neighborhood $U$ of the origin. As before, one can distinguish between \emph{existential} and \emph{constructive} local counting problems.
 \item Solution of a constructive local counting problem allows to compute \emph{semiglobal bounds} for the number of intersections in any compact subset $K\subseteq\C^n$ of known size by covering $K$ by neighborhoods of sizes bounded from below. However, the space $\C^n$ is non-compact. Asking about the total bound for the number of isolated  intersections is the global counting problem.
\end{enumerate}

\begin{Ex}
Theorem~\ref{thm:meandering} provides a solution for the local (hence for the semiglobal) problem of counting intersections between trajectories of a polynomial vector field $v$ in $\R^n$ and an algebraic hypersurface $\varPi$. It gives explicit bounds for the number of isolated intersections between integral curves of bounded size $R$ (in the $(t,x)$-space time) and $\varPi$ which for autonomous vector fields implies the answer in terms of the distance from the singular locus $\operatorname{Sing}(v)=\{x:v(x)=0\}\subset\R^n$ even if some integral curves lie completely on $\varPi$.

Yet this bound is not global: the bound grows to infinity as the size $R\to+\infty$ grows unbounded.
\end{Ex}

We start this section with bounds for multiplicity, namely, the bound by Gabri\`elov and Khovanskii which improves the double exponential bound implied by Theorem~\ref{thm:meandering} to a single exponential one and stresses the difference of the roles of degree and dimension.

\subsection{Infinitesimal version: multiplicity counting}
In this section we study how the Rolle-type technique can be generalized to the local (i.e., tuned to the local ring) in the multidimensional case.

\subsubsection{Multiplicity of maps in one variable}\label{sec:mult-1}
Let us return for a moment to the framework of \secref{sec:rolle-1}, but for simplicity consider holomorphic functions of one variable $z$ defined, say, in the unit disk $\D\subseteq\C$ and continuous on its boundary. By default we consider only functions that do not vanish identically on $\D$.

The multiplicity $\mult_0 f$ of $0\not\equiv f\in\O(D)$ at $z=0$ can be defined by several ways. First, we can say that $\mult_0f$ is equal to a finite natural number $k\ge 1$, if
\begin{equation}\label{mult-op-1}
 f(0)=f'(0)=\cdots=f^{(k-1)}(0)=0,\qquad f^{(k)}(0)\ne 0.
\end{equation}
This gives an easy way to compute $\mult_0 f$ by evaluating several differential operators on $f$ and checking that they all vanish at the point $z=0$.

\begin{Small}
\begin{Rem}[terminological]
The multiplicity of a root is very closely related to the \emph{order of tangency}: a root of multiplicity $k\ge 1$ corresponds to tangency of order $k-1$ between the graph of the function $f$ and the  $z$-axis in $\D\times\C$. Simple root corresponds to the transversal intersection and by convention corresponds to tangency of zero order.
\end{Rem}
\par\end{Small}

The geometric definition of multiplicity is as follows: replace $f(z)$ by $f(z)-\e$, where $\e$ is a very small complex number. Assume that all roots of $f-\e$ near the origin are simple. Then $\mult_0 f$ is equal to the number of such roots\footnote{To be accurate, one has first to choose a small disk $\delta D$ of radius $\delta>0$ such that it contains no nonzero roots of $f$, in particular, $|f|$ is bounded from below by some $s>0$ on the boundary $\{|z|=\delta\}$. Then we can choose any \emph{regular} value $\e$ of $f$ small enough that $|\e|<s$ and count the roots of $f-\e$ in $\delta D$. The Rouch\'e theorem takes care of the consistency of this process.}.

Finally, the algebraic definition of multiplicity is given in terms of the ring (a $\C$-algebra) of holomorphic germs $\O(\C,0)$ at the origin. Any function $f$ defines the principal ideal $I=I_f=\<f\>=f\cdot\O(\C,0)\subset\O(\C,0)$ which has finite codimension $k$ (dimension of the quotient \emph{local algebra}),
$$
 k=\dim_\C \O(\C,0)/I_f<+\infty
$$
if and only if $f$ has an isolated root at the origin (i.e., not identically zero on the one-dimensional case).

An obvious inspection shows that all three definitions are equivalent and give the same answer, $\mult_0 f=\ord_0 f$, where the order $\ord_0 f$ is the number of the first nonzero coefficient in the Taylor expansion
$$
f(z)=a_0+a_1x+a_2z^2+\cdots+a_k z^k+\cdots.
$$
One should understand the multiplicity as the number of small simple roots that collided/coalesced to form a degenerate multiple root at the origin. The following claim is obvious.

\begin{Prop}\label{lem:locRolle}
$
\mult_0 f \le \mult_0 f'+1.
$
\end{Prop}
This is fully in the spirit of the Rolle inequality from Proposition~\ref{prop:real-rolle}. As in the case of the ``genuine'' Rolle inequalty, it can be strict.

\begin{Cor}\label{cor:k local}
If $\mult_0 f\le k$, then any small perturbation of $f$ cannot have more than $k$ simple roots near the origin.
\end{Cor}

This qualitative statement (requiring proper assignment of quantifiers, see the previous footnote), may be made completely explicit and quantitative. If a function $f\in\O(\overline D)$ has more then $k$ zeros in a sufficiently small disk $\delta D$, then its $k$th derivative must be very small.

In order to formulate bounds invariant by the rescaling $f\mapsto \l f$, $0\ne \l\in\C$, we will impose a normalizing condition, say,
\begin{equation}\label{norm}
\|f\|_D=\max_{z\in D}|f(z)|=1.
\end{equation}

\begin{Thm}[\cite{multi}]\label{thm:lowerbound}
If $f$ normalized as in \eqref{norm} has a nonzero $k$th derivative $|f^{(k)}(0)|=s>0$, then it has no more than $k$ isolated zeros in the disk $rD$ of some radius $r>0$ proportional to $s$: $r\ge \delta_k s$, where the constant $\delta_k>0$ is explicit and depends only on $k$ and not on the function $f$ or any other parameter.

Away from these zeros the function admits an explicit lower bound for $|f(z)|>0$.
\end{Thm}

The proof of this result is based on classical estimates from complex and harmonic analysis, in particular, on the Jensen inequality, see \secref{sec:jensen} and is quite similar to Theorem~\ref{thm:equivalence}. The last assertion can be made precise in terms of the natural parameters (the distance from zeros e.a.). The important fact is that the bound does not depend on $f$ except through $s$ and $k$.

\subsubsection{Noetherian multiplicities: the isolated case}\label{sec:GK-mult}
The technique developed in \cite{GK} allows to prove a multidimensional analog of Theorem~\ref{thm:GJ}. To formulate it, we introduce the class of Noetherian functions.

First we recall the definition of multiplicity.
We consider holomorphic germs
\begin{gather*}
F=(f_1,\dots,f_n):(\C^n,0)\to(\C^n,0),\quad f_i\in\O(\C^n,0)
\\
(z_1,\dots,z_n)=z\longmapsto F(z)=\bigl(f_1(z),\dots,f_n(z)\bigr).
\end{gather*}

For such germs the notion of multiplicity was developed in the context of Singularity Theory \cite{AVG1}.
\begin{Def}\label{def-multiplicity-n}
$$
\mult_0 F=\operatorname{codim} \<F\>=\dim_\C\O(\C^n,0)/\<f_1,\dots,f_n\> \le \infty
$$
and it is finite if and only if $F^{-1}(0)=\{0\}$ is isolated.
\end{Def}

Defined in such a manner, the multiplicity is equal to the number of preimages $F^{-1}(w)$, where $w$ is any regular (non-critical) value of $F$ sufficiently close to the origin $w=0$. Note that the order
$$
\ord F=\min_i \ord f_i,\qquad i=1,\dots,n,
$$
though well defined, is not equal to $\mult_0F$ for $n>1$ anymore (e.g., when $f_i(z)=z_i^k$, $i=1,\dots,n$).

\begin{Def}\label{def:Noetherian}
A subring $\mathscr N\subset\mathscr O(\C^n,0)$ of the ring of holomorphic germs is called the ring of Noetherian functions (in short, Noetherian ring), if it contains the germs of all polynomials from $\C[x]=\C[x_1,\dots,x_n]$, is closed by all partial derivations $\partial/\partial x_j$ and there exist a finite tuple of germs $\psi_1,\dots, \psi_m\in\mathscr N$ which generates $\mathscr N$ over $\C[x_1,\dots,x_n]$.
\end{Def}

The latter condition means that $\mathscr N=\C[x_1,\dots,x_n,\psi_1,\dots,\psi_m]$, where
\begin{equation}\label{N-chain}
\frac{\partial \psi_i}{\partial x_j}=P_{ij}(x,\psi)\in\C[x,\psi]\qquad \forall i=1,\dots,n,\ j=1,\dots,m.
\end{equation}
The dimensions $n,m$ and the degree $\delta=\max\deg P_{ij}$ are natural characteristics of any Noetherian ring: to stress this, we will use the notation $\mathscr N=\mathscr N_{n,m;\delta}$. The tuple of functions $\psi$ will be referred to as the Noetherian chain.

Any Noetherian ring is naturally filtered: for any $\f=Q(x,\psi)\in\mathscr N$ its degree $d=\deg_{x,\psi}Q$ is well defined (there can be algebraic dependencies, so the representation of a minimal degree should be considered).

It turns out that Noetherian rings possess a B\'ezout-like property: the multiplicity of any isolated intersection is explicitly bounded. Recall that the original B\'ezout theorem claims that an isolated solution of a system of polynomial equations of degree $\le d$ in $\C^n$ does not exceed $d^n$, the bound that is polynomial in the degree and exponential in the dimension.

\begin{Thm}[A. Gabri\`elov and A. Khovanskii, \cite{GK}]\label{thm:GK-n}
If $\f_1,\dots,\f_n\in\mathscr N_{n,m;\delta}\subseteq\O(\C^n,0)$ are germs from a Notherian ring and the \emph{complete} intersection $X_0$,
$$
\f_1(x)=0,\ \dots,\ \f_n(x)=0,\qquad x\in(\C^n,0),
$$
is isolated, being the origin in $\C^n,0)$, then the multiplicity of this intersection $\mu=\mult_0 \f<+\infty$ is explicitly bounded in terms of the parameters $n,m,\delta$ of the Notherian ring and the maximal degree $d=\max_i \deg\f_i$. The bound is polynomial in the degrees $d$ and $\delta$ and simple exponential in $n,m$.
\end{Thm}

\begin{Rem}\label{rem:nonisol}
In the assumptions of Theorem~\ref{thm:GK-n} all intersections $X_\e$ of the form
$$
\f_1(x)=\e_1,\ \dots,\ \f_n(x)=\e_n,\qquad x\in U,\ \e=(\e_1,\dots,\e_n)\in (\C^n,0),
$$
in a sufficiently small neighborhood $U$ of the origin will also be isolated points and their number $\#X_\e$ will not exceed $\mu$ by the definition of multiplicity.

What Theorem~\ref{thm:GK-n} does not allow to do is to place an upper bound in the case where the original intersection $X_0$ is non-isolated.
\end{Rem}

\subsubsection{Multiplicity operators in the multidimensional case}
How the results of section~\ref{sec:mult-1}, in particular, Theorem~\ref{thm:lowerbound} can be generalized for the case of functions of several variables?

What is missing in the multidimensional case is the description of germs of given finite multiplicity $k$ by zeros of differential operators as in \eqref{mult-op-1}. This is the gap that we will close now, following \cite{multi}.

The basic observation is that \emph{having multiplicity of order $k$ and less} depends only on the Taylor terms of $F$ of order $k$ and less, and this dependence is algebraic (i.e., expressed by a number of polynomial equalities in the Taylor coefficients of $F$ of order $\le k$), that is, an algebraic set in the jet space $J=J_k$. In the one-dimensional case this is completely transparent: if $f(z)=c_0+c_1z+\cdots+ c_kz^k+\ldots$ with $c_k\ne 0$ (algebraic identity guaranteeing that $\mult_0 f\le k$), then any higher order terms cannot change this fact.

%

More precisely, the above implies that the condition $\mult_0\<f_1,\dots,f_n\>\le k$   is equivalent to the condition that the map
\begin{equation*}
	E\colon J^n\to J,\quad E(a_1,\ldots,a_n)=j^k\biggl(\sum a_if_i\biggr)
\end{equation*}
has corank at most $k$. Choosing a standard monomial basis of $J$, this amounts to checking if some of the minors of size $\dim J-k$ of a matrix whose entries are Taylor coefficents of $f_i$, are non-vanishing.

As  Taylor coefficents of $f_i$ are (up to a constant) partial derivatives of $f_i$ evaluated at $0$,  these minors  are equal to $ M^{(k)}_\beta F(0)$, where $\beta$ is the multiindex of the corresponding minor and $ M^{(k)}_\beta F$ are  some universal polynomial homogeneous differential operators $ M^{(k)}_\beta$ of degree $\dim J-k$ and of order $\le k$  evaluated on the tuple $F=(f_1,\ldots,f_n)$.

This proves the following result.

\begin{Thm}\cite{multi}\label{multi-op}
\begin{equation}
\mult_0 F> k\Longleftrightarrow M^{(k)}_\beta F(0)=0 \quad\forall \beta\in B,
\end{equation}
where $M^{(k)}_\beta$ are differential operators of order $\le k$ indexed by a finite set $B$.
\end{Thm}

The operators $M^{(k)}_\beta$ (or the entire collection $M^{(k)}_B$) are called the \emph{multiplicity operators}.
They generalize the collection of operators
$$
1, \tfrac{\mathrm d}{\mathrm dx}, \tfrac{\mathrm d^2}{\mathrm dx^2}, \dots, \tfrac{\mathrm d^k}{\mathrm dx^k}
$$
for the case of one variable $x=z_1$ when $n=1$ (cf.~with \eqref{mult-op-1}).

\subsubsection{Lower bounds}
The key fact is that the above relation can be quantified: values  of $M^{(k)}_\beta F$ both bound from below the "distance" to the set of parameters such that $\mult_0F>k$ and control  geometry of  the zero set, completely similar to the univariate case  of Theorem~\ref{thm:lowerbound}.
The construction of multiplicity operators does not allow to write so simple formulas for $M^{(k)}_B$, but is explicit enough to derive a number of their properties. What is important to control is the order of these operators (which is at most $k$). Besides, the coefficients of these operators are defined over $\Q$ and their height (the maximal natural number required to represent them as irreducible fractions) is explicitly bounded in terms of $n$ and $k$. This follows from the fact that minors of a matrix are polynomials in the entries of this matrix with integer coefficients of bounded height.

This suffices to produce explicit upper bounds for the number of roots $\{F=0\}$ in small (compared to $1$) polydiscs and give explicit lower bounds for $|F(z)|=\max_i |f_i(z)|$ when $z$ is away from these roots. Syntactically (the order of quantifiers and the universal nature of the constants depending only on $n,k$) the results are parallel to Theorem~\ref{thm:lowerbound}, but the precise form of the inequalities is considerably more involved. Very roughly, if one of the multiplicity operators is bounded away from zero by some value $s>0$, then these lower bounds are linear in $s$.

\subsubsection{Rolle inequality for the multiplicity operators}
This inequality (not suprising at all) follows from the fact that the multiplicity operators have order at most $k$, but its formulation requires restricting the functions $f_i\in\O(\C^n,0)$ to arbitrary (real) analytic curves through the origin.

Let $\gamma:(\R_+,0)\to(\C^n,0)$ be a real analytic (vector) function parametrized by the (Hermitian) arclength $t\ge0$. Then for any function $g\in\O(\C^n,t)$ the order $\ord_\gamma g$ is defined as a nonnegative rational number,
$$
\ord_\gamma g=\lim_{t\to 0^+}\frac{\log|f\bigl(\gamma(t)\bigr)|}{\log t}.
$$
Rationality of the order follows from the well known fact that the coordinate functions of $\gamma$ are convergent Puiseaux series of $t$, and the order is the leading Puiseaux exponent.

This order is related to the multiplicity by the following simple observation.

\begin{Prop}
If $F=(f_1,\dots,f_n)$ and for all $i=1,\dots,n$ the orders satisfy the inequality $\ord_\gamma f_i\ge k\in\N$, then $\mult_0 F\ge k$. \qed
\end{Prop}

\begin{Cor}[Rolle inequality for the multiplicity operators]
For any real analytic curve $\gamma$ through the origin as above and any multiplicity operator $M^{(k)}$,
$$
\ord_\gamma M^{(k)}(F)\ge\min_{i=1,\dots,n} \{\ord_\gamma f_i\}-k.
$$
\end{Cor}

The above property implies a direct analogue of Proposition~\ref{lem:locRolle} for cofinite ideals: for any ideal $I\subset\O(\C^n,0)$ of finite codimension
\begin{equation*}
\left (\mult I\right )^{\frac 1 n}\le \left (\mult M_{n,k}(I)\right )^{\frac 1 n}+k,
\end{equation*}
where $M_{n,k}(I)$ is generated by  $I$ and all functions $M^{(k)}(f_1,\dots,f_n)$
with $M^{(k)}$ being all multiplicity operators of order $k$ and  $f_i\in I$.

\subsubsection{Application to Noetherian functions}
Construction and properties of multiplicity operators can be applied to counting not just the multiplicity of Notherian germs as in \secref{sec:GK-mult}, but also their number in a ball of controlled size, cf.~with Remark~\ref{rem:nonisol}.

Note that Definion~\ref{def:Noetherian} can be ``delocalized'' almost verbatim. A Noetherian ring of functions $\mathscr S$ in a domain $U\subseteq\C^n$ is an algebraic extension of the ring $\C[x_1,\dots,x_n]$ by a finite tuple $\psi=(\psi_1,\dots,\psi_m)$, $\psi_i\in\O(U)$ of functions satisfying the system of differential equations \eqref{def:Noetherian}. These functions are called the Noetherian chain generating $\mathscr S$.

Consider the germ $X\subset U$ of the Noetherian set
\begin{equation*}
\begin{gathered}
X=\{x\in U : \f_1(x)=\dots=\f_{n-1}(x)=0\},
\\
\f_i=P_i(x, \psi)\in\mathscr S,\qquad i=1,\dots,n-1
\end{gathered}
\end{equation*}
at a point $p\in X$, where $P_i(x, \psi)\in\C[x,\psi]$, $\deg P_i\le d$.

\begin{Def}
The \emph{deformation multiplicity}, or \emph{deflicity} of $X$ \emph{with respect to a Noetherian function} $\rho\in \mathscr S$ is the number of \emph{isolated} points in $\rho^{-1}(y)\cap X$ (counted with multiplicities) which converge to $p$ as $y\to\rho(p)$.
\end{Def}

\begin{Thm}[\cite{multi}]\label{thm:deflicity}
The maximum possible deformation multiplicity $\mathscr{M}(m,n,\delta,d)$ for any Noetherian system with parameters $m,n,\delta$ and any Noetherian set $X$ defined as above with $\deg \f_i\le d$ and with respect to any $\rho\in\mathscr S$, $\deg \rho\le d$, admits an effective upper bound
\begin{equation*}
\mathscr{M}(m,n,\delta,d)\le \left(\max\{d,\delta\}(m+n)\right)^{(m+n)^{O(n)}}.
\end{equation*}
\end{Thm}
An analogue of the above result for Pfaffian functions was proved by Gabrielov \cite{GabLojas}, and this allowed him in a series of joint works with Vorobjov to establish effective bounds on the complexity of subPfaffian sets, see \cite{GabVor} for details. An improvement of these result in \cite{BinVor} served as the main motivation to introduction of sharply o-minimal sets, see \cite{BinNovIMC}.
It is plausible that a parallel local theory of Noetherian sets exists, with the above result serving as the cornerstone of this theory.

\begin{Ex}
Let $\rho=x_n$, i.e. we count the number of isolated solutions of  a deformation of a system of equations $$ \psi_1(x_1,\dots, x_{n-1},0)=\dots=\psi_{n-1}(x_1,\dots, x_{n-1},0)=0$$ in $\{x_n=0\}$.

If $0$ is an isolated solution of this system then its multiplicity is bounded from above by Theorem~\ref{thm:GK-n}. Thus,  any holomorphic perturbation of this system will have only isolated solutions near $0$, and their number (counted with their multiplicities) will be equal to this multiplicity.

If $0$ is not isolated solution, then one can always find a holomorphic perturbation with any given number of isolated solutions converging to $0$. The requirement that $\psi_i(x_1,\dots, x_{n-1},x_n)$ are Noetherian functions of given complexity restricts the class of perturbations, thus allowing to give a meaningful bound.

A typical application is provided by a Lojasiewicz-type inequality for Noetherian functions (recall that the original Lojasiewicz inequality was for algebraic functions):
\begin{Cor}\label{cor:Lojas}
Let $f,g$ be real Noetherian functions of $n$ variables of degree $d$ defined by the same Noetherian chain, $f(p)=0$ and assume that $\{f=0\}\subset \{g=0\}$. Then there exists a constant $k$,
$$
0\le k\le \bigl(\max\{d,\delta\}(m+n)\bigr)^{(m+n)^{O(n)}}
$$
such that $|f|>|g|^k$ near $p$.
\end{Cor}
To prove it, one should consider the set $X=\{df\wedge dg=0\}$ of critical values of the restrictions of $|f|$ to the intersections of level curves of $g$ with a small ball around $p$. A standard argument shows that the deflicity of this set with respect to $g$ bounds $k$, see \cite{multi} for details. As $p$ is not an isolated point of $X$ in general,  Theorem~\ref{thm:GK-n} is not applicable. However, Theorem~\ref{thm:deflicity} is applicable, which implies the required bound.
\end{Ex}

\begin{Small}
\begin{proof}[Sketch of the proof]
An equivalent   definition of Noetherian functions is as follows:   consider a distribution $\Psi$ of codimension $m$ on  $\C^{n+m}_{x,\psi}$ defined by an $m$-tuple of one-forms
\begin{equation*}
\omega_i=d\psi_i-\sum_{j=1}^nP_{ij}(x,\psi) dx_j, \quad P_{ij}\in\C[x,\psi]\quad i=1,\dots,m,
\end{equation*}
and let $\Lambda_p$ be an $n$-dimensional integral surface of $\Psi$ (i.e. $\omega_i|_{\Lambda_p}\equiv0$) passing through $p\in \C^{n+m}_{x,\psi}$. The set $X$ is then naturally identified with the intersection $\Lambda_p\cap \{P_1=\dots=P_{n-1}=0\}$, and $\rho=Q|_{\Lambda_p}$ for some $Q\in\C[x,\psi]$.

Note that the existence of $\Lambda_p$ is a non-trivial condition for $n>1$, and the necessary condition  for its existence is given by the Frobenius theorem. Evidently, the integral surface is not algebraic in itself, but the union of all such surfaces is an algebraic subset $A\subseteq\C^{n+m}$.

We take the union $Y(P,Q)_p$ of all components of $X$ which are not curves (i.e., of higher dimension) or are inside $\rho^{-1}(\rho(p))$. This is not an algebraic set, but the union $Y(P,Q)=\cup_{q\in X}Y(P,Q)_q$ is also algebraic. Indeed these are points $q\in X$ where the intersection $\{Q=P_1=\dots=P_{n-1}=0\}$ with the integral surface  $\Lambda_q$ of $\Psi$ containing $q$ is not isolated.  This is equivalent to the condition that the multiplicity of this Noetherian intersection at $q$ is greater than the upper bound of Theorem~\ref{thm:GK-n}. The latter is equivalent to vanishing  of multiplicity operators $M^{(k)}_{\beta}(\f_1,\dots,\f_{n-1},\rho)$, which are equal to restriction to $\Lambda_q$ of some polynomials $S^{(k)}_{\beta}$ independent of $q$, so $Y(P,Q)=\cap\{S^{(k)}_{\beta}=0\}$.

The proof of Theorem~\ref{thm:deflicity} goes by induction on dimension of $Y(P,Q)$. On each step of induction we replace polynomials $P_i$ by polynomials $P'_i$ of bigger degree such that the deflicity doesn't drop, but the dimension of the set $Y(P,Q)$ decreases. Properties of multiplicity operators allow to control the degree of this perturbations, and the case $\dim Y(P,Q)=0$ is Theorem~\ref{thm:GK-n}.
\end{proof}
\par\end{Small}

%
%

\subsection{Local version in several dimensions: Binyamini theorem}
Theorem~\ref{thm:meandering} which solves the local counting problem for intersection between 1-dimensional trajectories of polynomial vector fields and codimension 1 algebraic hypersurfaces can be generalized to some extent for the intermediate dimensions. The corresponding result is due to G.~Binyamini \cite{BiPi}.

Consider not one, but rather $m$ vector fields $\boldsymbol\xi=(\xi_1,\dots,\xi_m)$, $1<m<n$ in $\C^n$ which \emph{commute} between themselves and are linear independent at a generic point of $\C^n$. Together they define on $\C^n$ a singular foliation $\mathscr{F}$ with leaves of dimension $m$ outside of the singular locus $\S_{\mathscr{F}}$.
Each leaf $\mathscr{L}_p$ passing through a nonsingular point $p\in\C^n$ can be locally parameterized by a biholomorphic map $\varphi_p:(\C^m,0)\to(\mathscr{L}_p,p)$. Denote by $B_R$ the ball (or polydisk) of radius $R>0$, assuming that the space in which this ball lies and its center are clear from the context. Denote by $\mathscr B_R$ the image $\varphi_p(B_R)$ of the ball $B_R\subseteq\C^m$ centered at the origin: this set is a piece of the leaf of $\mathscr F$ through $p\in\C^n$ of a controlled i\emph{ntrinsic} size.

Assume that all vector fields $\xi_i$ are defined over $\Q$ and denote by $\mathfrak{d}(\boldsymbol\xi)$ the maximum of the degrees $\deg\xi_i$ and their \emph{logarithmic} heights\footnote{These are the simplest settings: in \cite{BiPi} fields defined over algebraic numbers are considered on equal footing: the logarithmic height of an algebraic number $x$ with the minimal equation $a_0\prod_1^d (x-x_i)\in\Z[x]$ is by definition $d^{-1}\bigl(\log|a_0|+\sum_1^d \log^+|x_i|\bigr)$, $\log^+=\max(\log, 0)$. The logarithmic weight is the natural scaling for the usual weight we used for the rational numbers $\Q$, which allows to formulate succinctly various bounds: very roughly, dependence on degrees and the \emph{logarithmic} weight are comparable.}. In a similar way, $\mathfrak{d}(V)$ will be used for the maximum of degrees of algebraic subvarieties  $V\subset\C^n$ and their logarithmic heights if these subvarieties are defined over $\Q$.

\begin{Def}
Let $\boldsymbol\xi$ is an $m$-tuple of vector fields in $\C^n$ defining a foliation $\mathscr F$ and $V$ an algebraic subvariety in $\C^n$. The \emph{unlikely intersection locus} $\S_{\mathscr F,V}$ is the union $\S_{\mathscr{F}}\cup\{x\in\C^n:\dim (\mathscr L_x\cap V) > m-\operatorname{codim} V\}$.
\end{Def}
In the case $m=1$ this set consists of singular locus of the single vector field $\xi$ and the union of integral trajectories of $\xi$ entirely belonging to an algebraic hypersurface $V$.

\begin{Thm}[G. Binyamini \cite{BiPi}]\label{thm:BiPi}
The number of isolated intersections $\#\{\mathscr B_R\cap V\}$, counted with their multiplicities, is bounded from above by an explicit polynomial depending on $\mathfrak{d}(V),\mathfrak{d}(\boldsymbol\xi)$, the ``radius'' $R$ and the distance between a larger piece $\mathscr B_{2R}$ from the unlikely intersection locus,
\begin{equation}\label{BiPi}
\#\{\mathscr B_R\cap V\}\le \operatorname{Poly}\bigl(\mathfrak{d}(V),\mathfrak{d}(\boldsymbol\xi),\log R, \log \operatorname{dist}^{-1}(\mathscr B_{2R},\S_{\mathscr F,V})\bigr).
\end{equation}
\end{Thm}

\begin{Rem}
This result should be considered against the background of Theorems~\ref{thm:meandering} and~\ref{thm:GK-n}. In Theorem~\ref{thm:GK-n} the bound is polynomial in the degrees of the polynomial equations (algebraic and rational), while in Theorem~\ref{thm:meandering} it can be made exponential in $d$, cf.~with Remark~\ref{rem:1-overQ}. Theorem~\ref{thm:GK-n} is infinitesimal (deals only with the maximal multiplicity), hence has no parameter analogous to $R$ or the height. Theorem~\ref{thm:meandering} is local like Theorem~\ref{thm:BiPi}, but deals only with 1-dimensional foliations, and the bound involves the parameter $R$, which also plays the role of a height. On the other hand, Theorem~\ref{thm:meandering} works for pieces of 1-dimensional leaves arbitrary close to the unlikely intersection locus $\S_{\mathscr F,V}$ and does not require the (only) vector  field spanning $\mathscr F$ to be defined over $\Q$ (this is achieved by introducing the fictitious variables).

In both Theorems~\ref{thm:meandering} and~\ref{thm:GK-n} the bounds depend most crucially on the dimension $n$ of the problem (it is double exponential in one case and simple exponential in the other). In Theorem~\ref{thm:BiPi} the polynomial in the right hand side of \eqref{BiPi} of course depends on $n$, but the nature of this dependence remains unspecified (the polynomial can be explicitly computed by an algorithm).
\end{Rem}

It should be noted that the main body of the paper \cite{BiPi} deals with problems coming from the Diophantine geometry in the spirit of the Pila--Wilkie theorem \cite{PiWi} and its various generalizations.

\section{From local to global}
Theorem~\ref{thm:lode-voor} (combined with the standard argument principle) and Theorem~\ref{thm:roitman} allow to place explicit upper bounds on the number of zeros of solutions of linear differential equations \eqref{hlode} with \emph{meromorphic} coefficients $a_j(t)$ \emph{locally}, i.e., in an open set $U\Subset\C$, in the following two cases:
\begin{enumerate}
 \item Nonsingular case, where $U$ has no singularities in $U$, hence the coefficients are bounded by (explicit) constants, $|a_j(t)|\le A<+\infty$, $j=1,\dots,n$. In this case the answer is given in terms of $A$ and the size of the domain $U$.
 \item Fuchsian case, where $U$ has a unique Fuchsian singular point with only real characteristic numbers. In this case the equation can be reduced to the form \eqref{fode} (assuming that the singularity occurs at $t=0\in U$) with the coefficients $b_j$ holomorphic in $U$ and hence bounded there, $|b_j|\le B<+\infty$. The answer is given in terms of $B$ and the size of $U$.
\end{enumerate}
As was explained, these two cases are essentially the only ones where a finite bound could be expected: non-Fuchsian (wild) singularities and Fuchsian singularities with non-real characteristic numbers can easily\footnote{Consider, however, the result in  \cite{NovSha}.} have infinitely many zeros (consider, for instance, the Fuchsian singularity at $t=0$ whose solutions is the function $\frac12(t^{\mathrm i}+t^{-\mathrm i})=\cos\ln t$).

    \subsection{Quasialgebraic functions}
Can these results be globalized for equations \eqref{hlode} with \emph{rational} coefficients defined \emph{globally} on the \emph{complex} projective line $\P=\P^1$? Such equations should have only Fuchsian singularities on $\P$, including the point $t=\infty\in\P\ssm\C$ and have only real spectra (collections of characteristic numbers) at all these points.

The model case is that of algebraic functions of one or several variables. These can be considered as multivalued analytic functions ramified over the discriminant sets  (where the corresponding defining equations have multiple roots), having a controlled growth near their singularities, including those at infinity. The monodromy of algebraic functions consists in permutation of their branches, hence its matrices have only roots of unity as eigenvalues. It turns out that there is a broader class of functions which possess similar explicit finiteness property, namely admits global bounds for the number of their isolated roots. These are solutions of certain integrable Pfaffian equations.  Such functions are called quasialbraic functions ($Q$-functions for short) in \cite{banach}.

In the case of Fuchsian equations with rational coefficients the answer (an explicit bound for the number of isolated zeros) should naturally depend on the number $d\ge 2$ of such singularities  (the minimal number $d=2$ is realized by the Euler equation (Example~\ref{ex:charnos}). Given this number, the Fuchsian equations with the given number of singularities can be parameterized by an open subset of a large projective space $\P^q$, $q=q(n,d)<+\infty$. Indeed, any such equation can be reduced to the form
\begin{equation}\label{global-fuchs}
a_0(t)y^{(n)}+a_1(t)y^{(n-1)}+\cdots+a_{n-1}(t)y'+a_n(t)y=0,\qquad t\in \P,
\end{equation}
with \emph{polynomial} coefficients. The Fuchs conditions at each of the $\le d$ singular points imply that
\begin{equation}\label{poly-a}
a_j\in\C[\,t\,],\quad j=0,\dots,n,\qquad \gcd(a_0,\dots,a_n)=1,\ \deg a_0\le nd,
\end{equation}
and explicitly (in $n,d$) constrain degrees of other coefficients. Note that among equations \eqref{global-fuchs} there are also non-Fuchsian equations that occur when one or more singular points (roots of $a_0$ merge together. Besides, the spectral condition imposes certain semialgebraic conditions on the coefficients $a_j$. All this means that the set of Fuchsian equations with the required spectra is a fairly involved semialgebraic subset $\mathscr F$ in $\P^q$.

The number of isolated zeros of solutions considered as a function $N(\cdot):\P^q\to\N\cup \{+\infty\}$ (to be accurately defined  later in a special case) takes finite values on this subset, but is clearly unbounded there (heuristically, equations with large non-principal coefficients have very oscillating solutions). Yet one can hope that there is some control over how fast the counting function grows near the boundary of $\mathscr F$, cf.~with \cite{BYa}.

\subsubsection{Isomonodromic families}
Consider a system of linear ordinary first order differential equations with rational coefficients of the form
\begin{equation}\label{sys}
\frac{\d x_i}{\d t}=\sum_{j=1}^n a_{ij}(t)\,x_j,\ i=1,\dots,n,  \quad A(t)=\bigl\{a_{ij}(t)\bigr\}_{i,j=1}^n, \ a_{ij}\in \C(t).
\end{equation}
Such a system can be in a standard way reduced to scalar ordinary differential equations: each dependent variable (component) $x_i(t)$ satisfies an $n$-th (or smaller) order linear homogeneous equation, and they can be ``combined'': there exists an equation of order $\le n^2$ satisfied by all functions $x_1(t),\dots, x_n(t)$ \emph{simultaneously}. Note that the system \eqref{sys} may be regular and non-Fuchsian\footnote{A system of first order linear ordinary differential equations with meromorphic coefficients is called Fuchsian at a point $t_*\in\C$, if its coefficients $a_{ij}(t)$ have at most a first order pole at $t_*$, see \cite{thebook}. A Fuchsian singularity is moderate, i.e., solutions of the system grow at most polynomially as $t\to t_*$, but the converse is not true in general. This is a very important difference between systems of linear ordinary differential equations of first order and scalar linear  higher order differential equations.}  (i.e., the matrix $A(t)$ may have poles of order greater than 1), but  if the solutions are growing at most polynomially, then all the aforementioned higher order \emph{equations}  will be automatically Fuchsian.

There is a Pfaffian analog of a system \eqref{sys}. Consider, say, a projective space $\P^m$ whose points will be denoted by $\l$ and an $(n\times n)$-matrix-valued 1-form $\Om(\l)$ on it, $\Om=\bigl\{\Om_{ij}(\l)\bigr\}_{i,j=1}^n$, where $\Om_{ij}\in\bigland ^1(\P^m)$ are rational 1-forms with a singular (polar) locus $\S\subset\P^m$. If $X=\bigl\{x_{ij}(\l)\bigr\}$ is a holomorphic $(n\times n)$-matrix valued function on $\P^m$, denote by $\d X$ its differential, matrix 1-form. Then the system of Pfaffian equations written in the matrix form as
\begin{equation}\label{pfaff}
\d X(\l)=\Om(\l) X(\l), \quad \l\in \P^m\ssm\S
\end{equation}
has a holomorphic nondegenerate solution $X(\l)$ off the singular locus $\S$ if and only if $\Om$  satisfies the (Frobenius) integrability condition
\begin{equation}\label{integrability}
\d\Om=\Om\land\Om.
\end{equation}
The solution $X(\lambda)$ is in general ramified over $\S$, and with any closed loop
$$
 \gamma\:[0,1]\to\P^m\ssm\S
$$
there is associated a \emph{monodromy operator} $M_\gamma\in\operatorname{GL}(n,\C)$ defined by the relation  $\Delta_\gamma X=XM_\gamma$, where  $\Delta_\gamma X$ is  the analytic continuation of $X(\lambda)$ along $\gamma$.  This transformation is a linear transformation of the space of solutions of \eqref{pfaff}.

For any projective line $\ell\simeq\P^1\subseteq\P^m$, not entirely belonging to the singular locus $\S$, the system \eqref{pfaff} can be restricted on $\ell$ and in any affine chart $t\in\C^1$ on it it will take the form of a system $\d X(t)=A(t)X(t)\,\d t$ of  linear ordinary differential equations of the first order as in \eqref{sys}. The matrix $A=A_\ell$ will depend on $\ell$, and the entire family of restrictions becomes parameterized by an algebraic (multidimensional) parameter $\ell$ ranging over the suitable Grassmann manifold (or an dense subspace of it).
$$
\frac{\d X}{\d t}=A_\ell(t)X(t),\qquad t\in\C, \ \ell\subset\P^m
$$
This family, as one can easily see\footnote{Isomonodromy means that for any loop defined up to the free homotopy and avoiding singularities of a specific system $A_\ell$, the monodromy operator associated with this loop will remain locally constant when the parameter $\ell$ changes continuously, are conjugate to each other in the linear group $\operatorname{GL}(n,\C)$. In particular, if all singularities were at the transversal intersections between $\ell$ and $\S$, all eigenvalues associated with small loops around the singular points, are locally constant.}, is \emph{isomonodromic}: this follows immediately from the integrability condition \eqref{integrability}.

\subsubsection{How to count zeros of multivalued matrix functions}
Solutions of a linear equation over a simply connected domain constitute a linear space (depending on the domain). If the equation is globally defined, then each loop avoiding singularities gives rise to a monodromy, an automorphism preserving this space. Thus when counting isolated zeros of solutions to such equations, we need to make the following choices:
\begin{enumerate}
 \item Choose a particular value of the parameter $\ell$, admissible in the sense that $\S_\ell=\ell\cap\S$ consists of isolated points (poles of the matrix $A_\ell(t)$);
 \item Choose an open simply connected domain $U\subseteq\C\ssm\S_\ell$, eventually having singularities on the boundary $\partial U$;
 \item Choose a \emph{nontrivial} linear combination of solutions
$$
f(t)=\sum\limits_{i,j=1}^n c_{ij}x_{ij}(t),\qquad t\in U,\quad c_{ij}\in\C.
$$
\end{enumerate}
\begin{Rem}
For some reasons, the domain $U$ \emph{should not be spiraling} around a singular point: in some sense, it should belong to a single leaf of the Riemann surface of $f$. To exclude the spiraling patterns, we will assume that $U$ is a \emph{conformal triangle}, a domain bounded by three circular or straight line arcs. Clearly, any tame non-spiraling simply connected domain can be triangulated into such conformal triangles, but the number of pieces will grow if the domain crosses many different leaves of the Riemann surface of $f(t)$.

Then one can check whether the number of isolated zeros of $f$ in $U$ is finite (apriori they may accumulate to a singular point on the boundary $\partial U$) and, if indeed finite, try to prove that this number, which apriori depends on the choices of $\ell, U, \|c_{ij}\|$ made above, is uniformly bounded over all these choices.
\end{Rem}

\begin{Def}\label{def:bound-roots}
If the above upper bound can be chosen uniformly over all choices of the line $\ell$, the conformal triangle $U$ and the linear combination $f$, we will call it the \emph{bound for roots} and denote it $\mathscr N(\Omega)$. For brevity, when such uniform upper bound does not exist, we will write $\mathscr N(\Om)=+\infty$.
\end{Def}

\subsubsection{Quasiunipotent integrable system}
Recall that a linear automorphism of a finite-dimensional space is called \emph{quasiunipotent}, if all its eigenvalues are the roots of unity (hence have modulus one). We will describe a class of Pfaffian systems \eqref{pfaff} with a special monodromy group. Denote as before the singular locus of the meromorphic 1-form $\Om$ by $\S$. Let
\begin{equation}
 \tau:(\C^1,0)\to(\P^m, a),\qquad a\in\P^m,\quad \tau(z)\notin\S\text{ for }z\ne0
\end{equation}
be the germ of a nonconstant holomorphic curve embedded in $\P^m$: the center of this curve may be on or off $\S$, but the image of any sufficiently small circle $\{|z|=\e>0\}\subset(\C,0)$ will be a loop avoiding $\S$. All such loops on the same curve $\tau$ are free homotopic to each other, and will be called \emph{small loops} centered at $a$.

\begin{Def}
The system \eqref{pfaff} is \emph{quasiunipotent}, if the monodromy operator associated with any small loop, is quasiunipotent.
\end{Def}
This definition clearly is independent of the choice of the loop up to a free homotopy (quasiunipotence is preserved by conjugacy in $\operatorname{GL}(n,\C)$) or even of the choice of orientation of the loop (the inverse matrix will also be quasiunipotent). The choice of parametrization is also not important, even if one considers curves of the form $\l=\tau(z^\mu)$, $\mu\in\N$ with nontrivial multiplicity $\mu>1$. On the other hand, the monodromy operators associated with arbitrary loops avoiding $\S$ may be not quasiunipotent.

If the center of a small loop $a$ is off $\S$, then the corresponding monodromy operator is identical and hence trivially quasiunipotent. If $a\in\S$ belongs to the smooth part of $\S$ and the curve $\tau$ is transversal to $\S$ at this point, then the monodromy along the small loops carried by $\tau$ can be computed through the residue of $\Om$ at this point. It turns out that this is enough to ensure that the system is quasiunipotent.

\begin{Thm}[Kashiwara theorem \cite{kashi}]\label{thm:kashi}
If the monodromy operators associated with small loops with centers only on the smooth part $\operatorname{reg}\S$ are quasiunipotent, then the monodromy along all small loops will be automatically quasiunipotent.
\end{Thm}

This deep theorem can be considered as a sort of the removable singularity principle from several complex variables. The non-smooth points $\operatorname{sing}\S=\S\ssm\operatorname{reg}\S$ constitute an algebraic variety of codimension $\ge 2$ in $\P^m$, and the claim is that the assertion valid for all points except for this small set, is valid for the ``exceptional'' points in $\P^m$ as well. The fact reflects some fundamental topological property of analytic hypersurfaces: their singularities can be resolved by finitely many blow-ups to complete intersections with the fundamental groups generated by commuting small loops with centers on the irreducible components.

\subsubsection{Finiteness theorems}\label{sec:fin-pfaff}
After these preparations one can formulate the general finiteness theorems for quasialgebraic functions. Both are proved in \cite{invent}.

\begin{Thm}[qualitative form]\label{thm:exist-fin}
Consider the singular Pfaffian system \eqref{pfaff} on the projective space $\P^m$ with the rational $n\times n$-matrix 1-form $\Om$ of degree $d$, of the form $\d X=\Om X$.

Assume that\textup:
\begin{enumerate}
 \item The form is integrable;
 \item The singularitiy on the polar divisor is moderate;
 \item The monodromy of the system is quasiunipotent.
\end{enumerate}
Then the bound for roots in the sense of Definition~\ref{def:bound-roots} for the system \eqref{pfaff} is finite, $\mathscr N(\Om)<+\infty$.
\end{Thm}

To formulate the quantitative form of this theorem which gives an explicit bound for roots, we need to make an extra assumption which introduces the last necessary parameter of the system \eqref{pfaff} which obviously must affect the value $\mathscr N(\Om)$. Recall that a rational Pfaffian matrix 1-form $\Om$ defined over $\Q$ in the sense of Definition~\ref{def:defQ} has a finite characteristic called the \emph{height}, the largest natural number required to write explicitly all rational coefficients of $\Om$.

\begin{Thm}[quantitative finiteness theorem]\label{thm:explicit-fin}
If in the assumptions of Theorem~\ref{thm:exist-fin} above an additional condition holds,
\begin{enumerate}
 \item[(4)] The system \eqref{pfaff} is defined over $\Q$ and the height of the matrix 1-form $\Om$ is $s$.
\end{enumerate}

Then the bound for roots $\mathscr N(\Om)$ is explicit function of the integer data, $m,n,d$ and $s$, and has the form
\begin{equation}
 \mathscr N(\Om)\le s^{2^{\operatorname{Poly}(n,m,d)}},   \qquad \operatorname{Poly}\in\Z[n,m,d],
\end{equation}
where $\operatorname{Poly}(n,m,d)$ stands for an explicit polynomial of low degree, say, a monomial $(n^4dm)^5$ with an explicitly bounded coefficient.
\end{Thm}

Clearly, this bound makes any practical sense so that there is no reason to struggle for an optimal expression implied by the proof: it relies, among others, on complexity of certain algorithms from real algebraic geometry (estimating the maximal diameter of \emph{bounded} semialgebraic sets in the Euclidean space, defined over $\Q$ and having known height). However, these bounds is in the same vein as quantitative bounds of semi-algebraic or Pfaffian geometry. This indicates that this class of functions could generate a \emph{sharply o-minimal structure}, see \cite{BinNovIMC} for details.

Theorem~\ref{thm:exist-fin} is in the spirit of the uniqueness theorem for (real analytic) functions and its distant generalzation, Gabri\'elov--Teissier theorem \cite{gabr}. The statement in principle could be proved using the same tools properly extended by the Fewnomial theory, cf.~with \cite{asik-abint}, where a baby version of Theorem~\ref{thm:exist-fin} can be found between the lines.

Theorem~\ref{thm:explicit-fin} is of a completely different nature; its proof in \cite{invent} relies upon explicit results on roots of functions defined by analytic ordinary differential equations as outlined in this paper, and heavily based on the Rolle theorem in different reincarnations.

    \subsection{Abelian integrals and The Hilbert Sixteenth}
Theorems formulated in \secref{sec:fin-pfaff} were in fact obtained in an attempt to solve the so called Infinitesimal Hilbert 16th Problem (on limit cycles of planar polynomial vector fields). In its original formulation the Hilbert problem stands as a grand challenge, but various relaxed versions of it had been around for quite some time, see \cites{montreal,centennial}. The infinitesimal version asks about the number of limit cycles which are born by polynomial parametric perturbation of Hamiltonian systems on the plane (the latter, being conservative, by definition have no \emph{limit} cycles, although they may have continuum of periodic orbits).
\begin{equation}\label{perturb}
 \dot x=\frac{\partial H(x,y)}{\partial y}+\e Q(x,y), \ \dot y=-\frac{\partial H(x,y)}{\partial x}-\e P(x,y),\quad \e\in(\R,0).
\end{equation}
In the first approximation (keeping only terms linear in $\e$) the condition necessary for a birth of a limit cycle is given by vanishing of the Poincar\'e--Pontryagin integral, an integral of a polynomial perturbation 1-form $P\,\d x+Q\,\d y$  on $\R^2$ along an algebraic oval $H=h$, the corresponding level curve of the Hamiltonian $H\in\R[x,y]$:
\begin{equation}\label{period}
I(h)=\oint\limits_{H(x,y)=h}P(x,y)\,\d x+Q(x,y)\,\d y.
\end{equation}
Such integrals (as functions of $h$) are called \emph{periods} and are most intriguing objects in the Algebraic Geometry and especially in the Number Theory. The Infinitesimal Hilbert 16th problem reduces to the following question: given explicit constraint $d$ on $\deg H, \max(\deg P,\deg Q)$, place an upper bound for the number of isolated zeros of the integral above on its natural domain of definition (the intervals of the axis $\R$ for which $\{H=h\}$ has a continuous family of ovals).

Periods are intrinsically constrained by the Pfaffian systems of equations known (in the different areas) as the \emph{Picard--Lefschetz--Gauss--Manin connection}. Here we briefly describe how this system looks like.

\subsubsection{Complexification, universalization}
Let $\C_d[x,y]$ be the $\C$-linear space of polynomials of degree $\le d$ from $\C[x,y]$. The  collection  of projective algebraic curves of degree $\le d$, i.e. closures of  affine curves $\{H(x,y)=0, \deg H\le d\}$ in $\C P^2$, is parameterized by projectivization $\P^m=P\C_d[x,y]$ of this space; the degenerate projective curves correspond to parameters lying on some algebraic hypersurface $\S_{top}\subset \P^m$,  but generically, i.e. for $\lambda\notin\S_{top}$, the closure of the level curve $\G_\l={\{H_\l=0\}}\subseteq \C^2$ in  $\C P^2$ is a smooth algebraic curve (here $\l\in\P^m$ is a ``point'' representing the projective class of $H_\l\in\C_d[x,y]$).  The real ovals  represent  homological 1-cycles in the first homology group $H_1(\G_\l, \Z)$ of the generic curve of dimension  $n=\dim H_1(\G_\l, \Z)=(d-1)^2$. For any  smooth curve $\G_\l$ we can choose a basis of cycles $\delta_1(\l),\dots,\delta_n(\gamma)\in H_1(\G_\l,\Z)$, and this basis can be extended by continuity to all nearby curves $\G_{\l'}$ for $\l'$ sufficiently close to $\l$ in a unique way. This induces a canonical isomorphism of the first homology spaces $H_1(\G_\l,\C)$ and $H_1(\G_{\l'},\C)$, or, equivalently, a flat connection (the so called Gauss-Manin connection) in the associated vector bundle $\cup_{\l\notin\S_{top}}H_1(\G_\l,\C)\to \P^m\setminus\S_{top}$. The above isomorphism cannot be made global (i.e. this connection doesn't provide a trivialization of this bundle) due to presence of a non-trivial topological monodromy.

The homogeneous coordinates of the vector $\l$ can be identified with the coefficients of the polynomial $H$, hence the union $\cup_{\l\in\P^m}\{H=0\}\times\{\lambda\}\subset\C^2\times\P^m$ is defined over $\Q$ with height 1.

\begin{Rem}
Since the curve $\G_\l$ is affine, its homology is generated, besides the ``large cycles'' of its projective compactification in $\P^m$, by the small loops along the punctures which appear on the intersection with the infinite line. These small loops depend nicely on the parameters $\l$ as long as the curve remains transversal to the infinite line.
\end{Rem}

Now consider all polynomial 1-forms on $\C^2$. It is well known that one can find a monomial basis $\omega_1,\dots,\omega_n\in\varLambda^1(\C^2)$ in the space of 1-forms, which would be dual to a basis in the homology of a typical fiber $\G_\l$.

This construction results in the period matrix $X_\l=X(\l)$, $X(\l)_{ij}=\int_{\delta_i}\omega_j$, with the following properties:
\begin{enumerate}
 \item $X(\lambda)$ is a multivalued matrix function on $\P^m\ssm\S$, defined outside the singular locus $\l\in\S$, $\S_{top}\subset\S$, which corresponds to curves which are
  \begin{enumerate}
   \item singular (i.e., carrying non-smooth points, generically transversal self-integrsections),
   \item non-transversal to the infinite line $\P^2\ssm\C^2$,
   \item atypical for the given choice of basis, i.e. when the  selected forms $\omega_1,\dots,\omega_n$ become ``accidentally'' linear dependent on $\G_\l$.
  \end{enumerate}
 \item The matrix function $X(\l)$ has the monodromy defined by the topology of fibration of smooth algebraic curves of a given degree: after continuation along a path $\sigma\in\pi_1(\P^m\ssm\S, \cdot)$ avoiding $\S$ the cycles $\delta_i(\l)$ may undergo a permutation which corresponds to the right multiplication, $X(\l)\mapsto X(\l)\cdot M_\sigma$, for some \emph{constant} matrix $M_\sigma\in\operatorname{GL}(n,\C)$.
 \item Denoting by $\Delta_\sigma$ the result of analytic continuation along a path $\sigma$, we see that $\Delta_\sigma(\d X(\l)\cdot X^{-1}(\sigma)=\d X(\lambda)\cdot M_\lambda\cdot M_\lambda^{-1}=X^{-1}(\lambda)=\d X(\lambda)\cdot X^{-1}(\l)$. In other words, the logarithmic derivative of $X(\l)$ is a single-valued matrix 1-form  $\Om(\l)$ having at most a pole on $\S$. Such function is necessarily rational on $\l\in\P^m$.
 \item The forms $\omega_i$ are defined over $\Q$ and have height $1$.
\end{enumerate}

These observations already show that the period matrix satisfies the (integrable by construction) Pfaffian system \eqref{pfaff}. The regularity of $X(\l)$ as $\l$ approaches  the singular locus $\S$ follows easily as   the forms $\omega_i$ are polynomial (independent of $\l$) and the size of the cycle may grow at most polynomially when $\l\to\S$.
This implies that $\Omega$ has at most poles on $\S$, so is rational by GAGA: a single-valued function with moderate singularities is necessarily rational.

An upper bound for the degree $\deg Q$ can be easily derived from the estimates of the growth of the period $I(\l)$ when $\l$ approaches the singular locus $\S$.

As for the assumption that the logarithmic derivative $\Om=\d X\cdot X^{-1}$ is defined over $\Q$, we need to resort to the explicit Gelfand--Leray formula for derivation of periods, which reduces the problem of finding linear dependence between monomial 1-forms and their Gelfand--Leray derivatives. Since both are expressed as integrals of rational 1-forms over the same cycles, the linear dependence will be automatically over $\Q$ and its complexity easily bounded.

\subsubsection{Quasiunipotent monodromy}
The assumption on the quasiunipotence of the monodromy of the system of $\Omega$ follows by virtue of the Kashivara Theorem~\ref{thm:kashi} from the Picard--Lefschetz formula. Indeed, a generic (codimension one) point $\l\in\S$ on the singular locus corresponds to
\begin{enumerate}
 \item a curve exhibiting a Morse singularity in the affine part $\C^2\subset\P^2$,
 \item a smooth curve $\G_\l$ such that the restrictions $\omega_i|_{\G_\l}$ become linearly dependent,  or
 \item a curve  $\G_\l$ tangent to the infinite line $\P^2\ssm\C^2$.
\end{enumerate}
In the first case the monodromy is given by Picard--Lefschetz formula, so is unipotent. In the second case the monodromy is trivial. In the latter case the monodromy along the small loops around the puncture points is generated by a permutation of finitely many roots and hence is a root of unity.

%

\subsubsection{Explicit solution of the Infinitesimal 16th problem}
Assembling results of the two last sections together, we obtain an explicit bound for the number of isolated zeros of real periods.

\begin{Thm}
Let $H\in\R[x,y]$ be a real Hamiltonian in two variables of degree $\le d$, $\G_h$ a continuous family of real ovals $\G_h\subseteq\{H(x,y)=h\}\subseteq\R^2$ defined on a finite or infinite interval of regular values of $H$, and $\omega=P\,d x+Q\, dy$ a polynomial 1-form of degree $\le d$.

Then the period function $I(h)$, defined in \eqref{period}, may have at most finitely many isolated zeros on this interval, their number $\mathscr N(H,\omega)$ being uniformly bounded by a double exponential expression
\begin{equation}\label{doublexp}
\mathscr N(H,\omega) \le 2^{2^{\operatorname{Poly(d)}}}, \qquad \deg\operatorname{Poly(d)}\le 61.
\end{equation}
\end{Thm}

As before, there is no reason to believe that this bound is realistic, hence no reason to strive for the optimal bound for the degree or the explicit value of the coefficients.

\begin{Rem}
If the Poincar\'e--Pontryagin integral \eqref{period} vanishes identically, the perturbation
can still produce limit cycles that could be tracked by using higher Melnikov functions which will again be periods of more complicated forms, in general of degree higher than $d$. Yet the same Theorem~\ref{thm:explicit-fin} can be applied to these periods as well, giving an explicit bound (growing with the order $r$ of the Melnikov function), see \cite{N-bendit}. The question of the maximal order which guarantees integrability of the perturbation \eqref{perturb} for all values of the parameter $\e$, is known as the Poincar\'e center--focus problem, which is open even in the simplest case where $H(x,y)=x^2+y^2$ (and hence all Melnikov functions are polynomials of growing degrees).
\end{Rem}

\bigskip

\subsection*{Acknowledgments}
We are staying on the shoulders of the late Vladimir Igorevich Arnold, who was at the source of much of this theory. Our gratitude goes to Yulij Sergeevich Ilyashenko who opened  us the way, to Askold Khovanskii, who was shining upon us as the guiding star in our journey and to Andrei Gabri\`elov with his enormous erudition and intuition. Our special thanks go to our close friend and colleague Gal Binyamini, who continues to realize our  hopes far and beyond.

The research of D.N.~was supported by the Israel Science Foundation (grant No.~1167/17)
	and by funding received from the Minerva Stiftung with the funds
	from the BMBF of the Federal Republic of Germany.

S.Y.~is the Gershon Kest Professor of Mathematics at the Weizmann Institute of Science.

\end{document}